\documentclass[a4paper,12pt]{article}

\usepackage{amscd}
\usepackage{amssymb,amsfonts,amsmath,amsthm}
\usepackage{verbatim}

\usepackage{amsmath}
\usepackage{amsthm}
\usepackage{amssymb}

\usepackage{latexsym}
\usepackage{array}
\usepackage{enumerate}

\numberwithin{equation}{section}

\usepackage{amsmath,amssymb,mathrsfs,amsthm}
\usepackage[all]{xy}
\usepackage{graphicx}
\usepackage{ulem}
\usepackage{ascmac}

\usepackage{mathtools}

\newcommand{\rHom}{\mathrm{RHom}}
\newcommand{\BDC}{{\mathbf{D}}^{\mathrm{b}}}

\newcommand{\Mod}{\mathrm{Mod}}
\newcommand{\Hom}{\mathrm{Hom}}

\newcommand{\CC}{\mathbb{C}}
\newcommand{\NN}{\mathbb{N}}
\newcommand{\RR}{\mathbb{R}}

\newcommand{\ZZ}{\mathbb{Z}}
\newcommand{\D}{\mathcal{D}}
\newcommand{\E}{\mathcal{E}}

\newcommand{\M}{\mathcal{M}}
\newcommand{\N}{\mathcal{N}}

\renewcommand{\SS}{\mathcal{S}}
\newcommand{\calT}{\mathcal{T}}

\newcommand{\Ker}{\operatorname{Ker}}
\newcommand{\Coker}{\operatorname{Coker}}

\newcommand{\id}{{\rm id}}
\newcommand{\supp}{{\rm supp}}

\newcommand{\Int}{{\rm Int}}

\newcommand{\Perv}{{\rm Perv}}
\newcommand{\Sol}{{\rm Sol}}

\newcommand{\tl}[1]{\widetilde{#1}}

\newcommand{\simto}{\overset{\sim}{\longrightarrow}}

\newcommand{\op}{\mbox{\scriptsize op}}
\newcommand{\SD}{\mathcal{D}}

\newcommand{\SO}{\mathcal{O}}
\newcommand{\SA}{\mathcal{A}}
\newcommand{\SM}{\mathcal{M}}
\newcommand{\SN}{\mathcal{N}}

\newcommand{\SE}{\mathcal{E}}
\newcommand{\SF}{\mathcal{F}}
\newcommand{\SG}{\mathcal{G}}
\newcommand{\SH}{\mathcal{H}}
\newcommand{\calI}{\mathcal{I}}
\newcommand{\SC}{\mathcal{C}}
\newcommand{\Conn}{\mathrm{Conn}}

\newcommand{\Modhol}{\mathrm{Mod}_{\mbox{\rm \scriptsize hol}}}
\newcommand{\Modrh}{\mathrm{Mod}_{\mbox{\rm \scriptsize rh}}}

\newcommand{\BDCcoh}{{\mathbf{D}}^{\mathrm{b}}_{\mbox{\rm \scriptsize coh}}}

\newcommand{\BDChol}{{\mathbf{D}}^{\mathrm{b}}_{\mbox{\rm \scriptsize hol}}}
\newcommand{\BDCrh}{{\mathbf{D}}^{\mathrm{b}}_{\mbox{\rm \scriptsize rh}}}
\newcommand{\BDCmero}{{\mathbf{D}}^{\mathrm{b}}_{\mbox{\rm \scriptsize mero}}}
\newcommand{\DChol}{\mathbf{D}_{\mbox{\rm \scriptsize hol}}}

\newcommand{\DD}{\mathbb{D}}
\newcommand{\Lotimes}[1]{\overset{L}{\otimes}_{#1}}
\newcommand{\Dotimes}{\overset{D}{\otimes}}
\newcommand{\Dboxtimes}{\overset{D}{\boxtimes}}
\newcommand{\Potimes}{\overset{+}{\otimes}}
\newcommand{\Pboxtimes}{\overset{+}{\boxtimes}}
\newcommand{\rhom}{{\bfR}{\mathcal{H}}om}
\newcommand{\rihom}{{\bfR}{\mathcal{I}}hom}
\newcommand{\Prihom}{{\bfR}{\mathcal{I}}hom^+}

\newcommand{\I}{{\rm I}}
\newcommand{\che}[1]{\check{#1}}
\newcommand{\var}[1]{\overline{#1}}
\newcommand{\BEC}{{\mathbf{E}}^{\mathrm{b}}}
\newcommand{\BECstb}{{\mathbf{E}}_{\rm stb}^{\mathrm{b}}}
\newcommand{\ZEC}{{\mathbf{E}}^{\mathrm{0}}}
\newcommand{\ZECmero}{{\mathbf{E}}^{\mathrm{0}}_{\mbox{\rm \scriptsize mero}}}
\newcommand{\BECmero}{{\mathbf{E}}^{\mathrm{b}}_{\mbox{\rm \scriptsize mero}}}
\newcommand{\Q}{\mathbf{Q}}
\newcommand{\q}{\mathbf{q}}
\newcommand{\bfr}{\mathbf{r}}
\newcommand{\bfl}{\mathbf{l}}
\newcommand{\EE}{\mathbb{E}}
 
\newcommand{\bs}{\backslash}

\newcommand{\bfR}{\mathbf{R}}
\newcommand{\bfL}{\mathbf{L}}
\newcommand{\bfD}{\mathbf{D}}
\newcommand{\rmR}{{\rm R}}
\newcommand{\rmE}{{\rm E}}
\newcommand{\rmD}{{\rm D}}
\newcommand{\rmt}{{\rm t}}
\newcommand{\bfE}{\mathbf{E}}

\renewcommand{\Re}{\operatorname{Re}}

\newcommand{\singsupp}{{\rm Sing.Supp}}

\newcommand{\sing}{{\rm sing}}
\newcommand{\reg}{{\rm reg}}
\newcommand{\sh}{{\rm sh}}
\newcommand{\RH}{{\rm RH}}
\newcommand{\ord}{\operatorname{ord}}

\newtheorem{theorem}{Theorem}[section]

\newtheorem{corollary}[theorem]{Corollary}
\newtheorem{lemma}[theorem]{Lemma}
\newtheorem{sublemma}[theorem]{Sublemma}
\newtheorem{proposition}[theorem]{Proposition}
\newtheorem{fact}[theorem]{Fact}
\newtheorem{notation}[theorem]{Notatiton}
\theoremstyle{definition}
\newtheorem{definition}[theorem]{Definition}
\theoremstyle{remark}
\newtheorem{remark}[theorem]{\sc Remark}

\title{$\CC$-Constructible Enhanced Ind-Sheaves
}

\author{ Yohei ITO
\footnote{Graduate School of Mathematical Science, The University of 
Tokyo, 3-8-1, Komaba, 
Meguro, Tokyo, 153-8914, Japan. 
E-mail: yitoh@ms.u-tokyo.ac.jp }}

\date{}

\begin{document}
\maketitle

\begin{abstract}
A. D'Agnolo and M. Kashiwara proved that
their enhanced solution functor induces a fully faithful embedding
of the triangulated category of holonomic $\D$-modules
into the one of $\RR$-constructible enhanced ind-sheaves.
In this paper,
we define $\CC$-constructible enhanced ind-sheaves
and show that the triangulated category of them is equivalent to
its essential image.
Moreover we show that there exists a t-structure
on it whose heart is equivalent to the abelian category of holonomic $\D$-modules.
\end{abstract}

\section{Introduction}
In 1984, the Riemann-Hilbert correspondence for regular holonomic $\D$-modules
was proved by M. Kashiwara \cite{Kas84}.
He established an equivalence of categories between the triangulated category $\BDCrh(\D_X)$
of regular holonomic $\D_X$-modules on a complex manifold $X$ and 
the one $\BDC_{\CC-c}(\CC_X)$ of $\CC$-constructible sheaves on $X$.
More precisely, we have functors
\[\xymatrix@R=5pt{
\BDCrh(\D_X)^{\op}\ar@<0.7ex>@{->}[r]^-{\Sol_X}\ar@{}[r]|-{\sim}
&
\BDC_{\CC-c}(\CC_X)\ar@<0.7ex>@{->}[l]^-{\RH_X}\\
\M \ar@{|->}[r]
&
\Sol_X(\M) := \rhom_{\D_X}(\M, \SO_X)\\
\RH_X(\SF) := \rhom(\SF, \SO_X^{\rmt})
&
\SF\ar@{|->}[l]
}\]
quasi-inverse to each other.
Here, $\SO_X^{\rmt}$ is the ind-sheaf of tempered holomorphic functions
(see \cite{KS01} for the definition).
The triangulated category $\BDC_{\CC-c}(\CC_X)$ has a t-structure
$\big({}^p\bfD^{\leq0}_{\CC-c}(\CC_X), {}^p\bfD^{\geq0}_{\CC-c}(\CC_X)\big)$
which is called the perverse t-structure.
Let us denote by $$\Perv(\CC_X) :=
{}^p\bfD^{\leq0}_{\CC-c}(\CC_X)\cap {}^p\bfD^{\geq0}_{\CC-c}(\CC_X)$$ its heart
and call an object of $\Perv(\CC_X) $ a perverse sheaf.
The above equivalence induces an equivalence of categories between
the abelian category $\Modrh(\D_X)$ of regular holonomic $\D_X$-modules
and the one $\Perv(\CC_X)$ of perverse sheaves.
The problem of extending
the Riemann-Hilbert correspondence to cover
the case of holonomic $\D$-modules with irregular singularities
had been open for 30 years.
In 2015, A. D'Agnolo and M. Kashiwara proved that
there exists an isomorphism of $\D_X$-modules
\[\M\simto\RH_X^{\rmE}\big(\Sol_X^{\rmE}(\M)\big)\]
for any holonomic $\D_X$-module $\M\in\BDChol(\D_X)$ \cite{DK16}.
Here, we set $\Sol_X^{\rmE}(\M) := \rihom_{\D_X}(\M, \SO_X^{\rmE})$,
$\RH_X^{\rmE}(K) := \rhom^{\rmE}(K, \SO_X^{\rmE})$ and
$\SO_X^{\rmE}$ is the enhanced ind-sheaf of tempered holomorphic functions
(see \cite{DK16} for the definition).
In particular, the enhanced solution functor $\Sol_X^{\rmE}$ induces a fully faithful embedding
\[\Sol_X^{\rmE} : \BDChol(\D_X)^{\op}\hookrightarrow\BEC_{\RR-c}(\I\CC_X)\] 
of the triangulated category $\BDChol(\D_X)$ of holonomic $\D$-modules
into the one $\BEC_{\RR-c}(\I\CC_X)$ of $\RR$-constructible enhanced ind-sheaves. 
Moreover, in \cite{DK16-2},
they gave a generalized t-structure
$\big({}^{\frac{1}{2}}\bfE_{\RR-c}^{\leq c}(\I\CC_X),
{}^{\frac{1}{2}}\bfE_{\RR-c}^{\geq c}(\I\CC_X)\big)_{c\in\RR}$ on $\BEC_{\RR-c}(\I\CC_X)$
and proved that
the enhanced solution functor induces a fully faithful embedding
of the abelian category $\Modhol(\D_X)$ of holonomic $\D_X$-modules
into ${}^{\frac{1}{2}}\bfE_{\RR-c}^{\leq 0}(\I\CC_X)\cap
{}^{\frac{1}{2}}\bfE_{\RR-c}^{\geq 0}(\I\CC_X)$.
On the other hand, T. Mochizuki proved that
the image of $\Sol_X^{\rmE}$ can be characterized  by the curve test \cite{Mochi16}.

In this paper, 
we define $\CC$-constructible enhanced ind-sheaves
and give a more explicit description of the essential image
of the enhanced solution functor
$\Sol_X^{\rmE} : \BDChol(\D_X)^{\op}\hookrightarrow\BEC_{\RR-c}(\I\CC_X)$
with them.
We say that an enhanced ind-sheaf $K\in\ZEC(\I\CC_X)$ is $\CC$-constructible if
there exists a complex stratification $\{X_\alpha\}_{\alpha\in A}$ of $X$
such that $\pi^{-1}\CC_{Z_\alpha\setminus D_\alpha}\otimes \bfE f_\alpha^{-1}K$
has a modified quasi-normal form along $D_\alpha$ for any $\alpha\in A$ (see, Definition \ref{def3.7}),
where $f_\alpha : Z_\alpha \to X$ is a complex blow-up of $\var{X_\alpha}$
along $\var{X_\alpha}\setminus X_\alpha$
and $D_\alpha := f_\alpha^{-1}\big(\var{X_\alpha}\setminus X_\alpha\big)$.
Then we show that the category consisting of them
is a full abelian subcategory of $\ZEC_{\RR-c}(\I\CC_X)$.
Let us denote by $\BEC_{\CC-c}(\I\CC_X)$ the full subcategory
of $\BEC_{\RR-c}(\I\CC_X)$ consisting of cohomologically $\CC$-constructible complexes.
Then the following result is the main theorem of this paper:
\begin{theorem}
For any $\M\in\BDChol(\D_X)$, the enhanced solution complex $\Sol_X^\rmE(\M)$
of $\M$ is a $\CC$-constructible enhanced ind-sheaf.
On the other hand, 
for any $\CC$-constructible enhanced ind-sheaf $K\in\BEC_{\CC-c}(\I\CC_X)$,
there exists $\M\in\BDChol(\D_X)$
such that $$K\simto \Sol_X^{\rmE}(\M).$$
In particular, we obtain an equivalence of categories
\[\Sol_X^{\rmE} : \BDChol(\D_X)^{\op}\simto \BEC_{\CC-c}(\I\CC_X).\]
\end{theorem}

Moreover we show that there exists a t-structure
on the triangulated category $\BEC_{\CC-c}(\I\CC_X)$
whose heart is equivalent to the abelian category $\Modhol(\D_X)$
of holonomic $\D$-modules as follows.
We set
\begin{align*}
{}^p\bfE^{\leq0}_{\CC-c}(\I\CC_X) &:=
\{K\in\BEC_{\CC-c}(\I\CC_X)\ |\ \sh_X(K)\in{}^p\bfD^{\leq0}_{\CC-c}(\CC_X)\},\\
{}^p\bfE^{\geq0}_{\CC-c}(\I\CC_X) &:=
\{K\in\BEC_{\CC-c}(\I\CC_X)\ |\ \rmD_X^{\rmE}(K)\in{}^p\bfE^{\leq0}_{\CC-c}(\I\CC_X)\}\\
&=
\{K\in\BEC_{\CC-c}(\I\CC_X)\ |\ \sh_X(K)\in{}^p\bfD^{\geq0}_{\CC-c}(\CC_X)\},
\end{align*}
where the pair $\big({}^p\bfD^{\leq0}_{\CC-c}(\CC_X),
{}^p\bfD^{\geq0}_{\CC-c}(\CC_X)\big)$
is the perverse t-structure on $\BDC_{\CC-c}(\CC_X)$,
$\sh := \alpha_Xi_0^!\bfR^{\rmE} : \BEC(\I\CC_X)\to\BDC(\CC_X)$
is the sheafification functor and
$\rmD_X^{\rmE}$ is the duality functor for enhanced ind-sheaves.
Then we obtain the second main theorem of this paper.
\begin{theorem}
The pair $\big({}^p\bfE^{\leq0}_{\CC-c}(\I\CC_X),
{}^p\bfE^{\geq0}_{\CC-c}(\I\CC_X)\big)$
is a t-structure on $\BEC_{\CC-c}(\I\CC_X)$
and
its heart $$\Perv(\I\CC_X) :=
{}^p\bfE^{\leq0}_{\CC-c}(\I\CC_X)\cap{}^p\bfE^{\geq0}_{\CC-c}(\I\CC_X)$$ is equivalent to the abelian category $\Modhol(\D_X)$
of holonomic $\D_X$-modules.
\end{theorem}
Moreover, the pair $\big({}^p\bfE^{\leq0}_{\CC-c}(\I\CC_X),
{}^p\bfE^{\geq0}_{\CC-c}(\I\CC_X)\big)$
is related to the generalized t-structure
$\big({}^{\frac{1}{2}}\bfE_{\RR-c}^{\leq c}(\I\CC_X),
{}^{\frac{1}{2}}\bfE_{\RR-c}^{\geq c}(\I\CC_X)\big)_{c\in\RR}$
on $\BEC_{\RR-c}(\I\CC_X)$ as follows
\begin{align*}
{}^p\bfE_{\CC-c}^{\leq 0}(\I\CC_X) &=
{}^{\frac{1}{2}}\bfE_{\RR-c}^{\leq 0}(\I\CC_X)\cap
\BEC_{\CC-c}(\I\CC_X),\\
{}^p\bfE_{\CC-c}^{\geq 0}(\I\CC_X) &=
{}^{\frac{1}{2}}\bfE_{\RR-c}^{\geq 0}(\I\CC_X)\cap
\BEC_{\CC-c}(\I\CC_X).
\end{align*}

\begin{remark}
\begin{itemize}
\setlength{\itemsep}{0pt}
\item[(1)]
We can describe the algebraic irregular Riemann-Hilbert correspondence
as similar to the analytic case.
See \cite{Ito20} for the details.
\item[(2)]
We also remark that
T.Kuwagaki introduced another approach to the irregular Riemann-Hilbert correspondence \cite{Kuwa18}.
\end{itemize}
\end{remark}

\section*{Acknowledgement}
I am grateful to Kiyoshi Takeuchi for many discussions at University of Tsukuba
and for giving many ideas.
I am also grateful to Takuro Mochizuki for kindly answering some questions
and for giving many advises, in Porto and Kyoto. 
I would like to thank Taito Tauchi for giving many comments.

\section{Preliminary Notions and Results}
In this section,
we briefly recall some basic notions
and results which will be used in this paper. 
\subsection{Generalized t-Structures}
First, let us recall the notion of (classical) t-structure from \cite{BBD}.
We say that a full subcategory $\SS$ of a category $\SC$ is strictly full
if it contains every object of $\SC$ which is isomorphic to an object of $\SS$.
\begin{definition}
Let $\cal{T}$ be a triangulated category.
A (classical) t-structure $(\calT^{\leq0}, \calT^{\geq0})$ on $\calT$
is a pair of strictly full subcategories of $\calT$ such that, setting
$$\calT^{\leq n} := \calT^{\leq0}[-n],\hspace{7pt} \calT^{\geq n} := \calT^{\geq0}[-n]$$
for $n\in\ZZ$, we have:
\begin{itemize}
\setlength{\itemsep}{0pt}
\item[(i)]
$\calT^{\leq0}\subset\calT^{\leq1},\hspace{3pt} \calT^{\geq1}\subset\calT^{\geq0}$,

\item[(ii)]
$\Hom_{\calT}(\calT^{\leq0}, \calT^{\geq1})=0$,

\item[(iii)]
for any $X\in\calT$ there exists a distinguished triangle
$$X_{\leq0}\to X\to X_{\geq1}\xrightarrow{\ +1\ }$$
in $\calT$ with $X_{\leq0}\in\calT^{\leq0}$ and $X_{\geq1}\in\calT^{\geq1}$.
\end{itemize}
\end{definition}
The full abelian subcategory $\calT^{\leq0}\cap\calT^{\geq0}$ is called the heart of the t-structure.

Let us recall the notion of generalized t-structure from \cite{Kas15}. 
\begin{definition}
A generalized t-structure $(\calT^{\leq c}, \calT^{\geq c})_{c\in\RR}$ on $\calT$
is a pair of families of strictly full subcategories of $\calT$ satisfying conditions (i)-(iv) below,
where we set 
$$\calT^{< c} := \bigcup_{c'<c}\calT^{\leq c'},\hspace{7pt}
\calT^{> c} := \bigcup_{c'>c}\calT^{\geq c'}\hspace{5pt} \mbox{for any $c\in\RR$.}$$

\begin{itemize}
\setlength{\itemsep}{0pt}
\item[(i)]
$\calT^{\leq c} = \bigcap_{c'>c}\calT^{\leq c'}$ and
$\calT^{\geq c} = \bigcap_{c'<c}\calT^{\geq c'}$
\hspace{5pt}
for any $c\in\RR$,

\item[(ii)]
$\calT^{\leq c+1} = \calT^{\leq c}[-1]$ and
$ \calT^{\geq c+1} = \calT^{\geq c}[-1]$
\hspace{5pt}
for any $c\in\RR$,

\item[(iii)]
$\Hom_{\calT}(\calT^{< c}, \calT^{> c})=0$
\hspace{5pt}
for any $c\in\RR$,

\item[(iv)]
for any $X\in\calT$ and any $c\in\RR$,
there are distinguished triangles in $\calT$
$$X_{\leq c}\to X\to X_{> c}\xrightarrow{\ +1\ }
, \hspace{17pt}
X_{< c}\to X\to X_{\geq c}\xrightarrow{\ +1\ }$$
with $X_{\ast}\in\calT^{\ast}$ for $\ast$ eqaul to
$\leq c,\hspace{3pt} > c,\hspace{3pt} < c$ or $\geq c$.
\end{itemize}
\end{definition}

Remark that the condition (iii) is equivalent to either of the followings:
\begin{itemize}
\setlength{\itemsep}{0pt}
\item[(iii)']
$\Hom_{\calT}(\calT^{\leq c}, \calT^{> c})=0$
\hspace{5pt}
for any $c\in\RR$,

\item[(iii)'']
$\Hom_{\calT}(\calT^{< c}, \calT^{\geq c})=0$
\hspace{5pt}
for any $c\in\RR$.
\end{itemize}
We also remark that we can consider a (classical) t-structure as a generalized t-structure.

\subsection{Ind-Sheaves}
Let us recall some basic notions on ind-sheaves.
References are made to Kashiwara-Schapira \cite{KS01} and \cite{KS06}. 
Let $M$ be a good topological space
(i.e., a locally compact Hausdorff space
which is countable at infinity and has finite soft dimension). 
We denote by $\Mod(\CC_M)$
the abelian category of sheaves of $\CC$-vector spaces on $M$
and by $\I\CC_M$ that of ind-sheaves on it.
Then there exists
a natural exact embedding $\iota_M : \Mod(\CC_M)\to\I\CC_M$.
We sometimes omit it.
It has an exact left adjoint $\alpha_M$, 
that has in turn an exact fully faithful
left adjoint functor $\beta_M$.
The category $\I\CC_M$ does not have enough injectives. 
Nevertheless, we can construct the derived category $\BDC(\I\CC_M)$ 
for ind-sheaves and the Grothendieck six operations among them. 
We denote by $\otimes$ and $\rihom$ the operations 
of tensor products and internal homs, respectively. 
If $f : M\to N$ be a continuous map, we denote 
by $f^{-1}, \rmR f_\ast, f^!$ and $\rmR f_{!!}$
the operations of the inverse image,
the direct image, the proper inverse image and the proper direct image, respectively. 
Note that $(f^{-1}, \rmR f_\ast)$ and 
$(\rmR f_{!!}, f^!)$ are pairs of adjoint functors.
We also set $\rhom := \alpha_M\circ\rihom$.

\subsection{Ind-Sheaves on Bordered Spaces}
We shall recall a notion of ind-sheaves on a bordered space and results on it.
For the results in this subsection, we refer to D'Agnolo-Kashiwara \cite{DK16}.
A bordered space is a pair $M_{\infty} = (M, \che{M})$ of
a good topological space $\che{M}$ and an open subset $M\subset\che{M}$.
For a locally closed subset $Z\subset M$ of $M$,
we set $Z_\infty := (Z, \var{Z})$.
A morphism $f : (M, \che{M})\to (N, \che{N})$ of bordered spaces
is a continuous map $f : M\to N$ such that the first projection
$\che{M}\times\che{N}\to\che{M}$ is proper on
the closure $\var{\Gamma}_f$ of the graph $\Gamma_f$ of $f$ 
in $\che{M}\times\che{N}$. 
The category of good topological spaces is embedded into that
of bordered spaces by the identification $M = (M, M)$. 

We define the triangulated category of ind-sheaves on 
$M_{\infty} = (M, \che{M})$ by 
\begin{align*}
\BDC(\I\CC_{M_\infty}) &:= 
\BDC(\I\CC_{\che{M}})/\BDC(\I\CC_{\che{M}\backslash M}).
\end{align*}
The quotient functor
\[\q : \BDC(\I\CC_{\che{M}})\to\BDC(\I\CC_{M_\infty})\]
has a left adjoint $\bfl$ and a right 
adjoint $\bfr$, both fully faithful, defined by 
\[\bfl(\q F) := \CC_M\otimes F,\hspace{25pt} 
\bfr(\mathbf{q} F) := \rihom(\CC_M, F). \]
Moreover they induce equivalences of categories
\begin{align*}
\bfl &: \BDC(\I\CC_{M_\infty})\simto
\{F\in\BDC(\I\CC_{\che{M}})\ |\ \CC_{M}\otimes F\simto F\},\\
\bfr &: \BDC(\I\CC_{M_\infty})\simto
\{F\in\BDC(\I\CC_{\che{M}})\ |\ F\simto\rihom(\CC_{M}, F)\},
\end{align*}
respectively.
It is clear that the quotient category 
$$\BDC(\CC_{M_\infty}) := 
\BDC(\CC_{\che{M}})/\BDC(\CC_{\che{M}\backslash M})$$
is equivalent to the derived category $\BDC(\CC_{M})$
of the abelian category $\Mod(\CC_M)$
and there exists an embedding functor 
$\BDC(\CC_{M_\infty}) \hookrightarrow \BDC(\I\CC_{M_\infty})$.
We sometimes write $\BDC(\CC_{M_\infty})$ for $\BDC(\CC_{M})$,
when considered as a full subcategory of $\BDC(\I\CC_{M_\infty})$.
For a morphism $f : M_\infty\to N_\infty$ 
of bordered spaces, 
we have the Grothendieck operations 
$ \otimes, \rihom, \rmR f_\ast, \rmR f_{!!}, f^{-1}, f^! $
(see \cite[Definitions 3.3.1 and 3.3.4]{DK16}).

Let $j_{M} : M_\infty \to \che{M}$ be the morphism of bordered spaces
given by the open embedding $M\hookrightarrow \che{M}$.
Actually, the functors
$j_M^{-1}\simeq j_M^! : \BDC(\I\CC_{\che{M}})\to\BDC(\I\CC_{M_\infty})$ 
are isomorphic to the quotient functor
and the functor $\bfR j_{M!!} : \BDC(\I\CC_{M_\infty})\to\BDC(\I\CC_{\che{M}})$
(resp.\ $\bfR j_{M\ast} : \BDC(\I\CC_{M_\infty})\to\BDC(\I\CC_{\che{M}})$)
is isomorphic to the functor $\bfl$ (resp.\ $\bfr$).
Then we have the following standard t-structure of $\BDC(\I\CC_{M_\infty})$
which is induced by the standard t-structure of $\BDC(\I\CC_{\che{M}})$:
\begin{align*}
\bfD^{\leq 0}(\I\CC_{M_\infty}) & = \{F\in \BDC(\I\CC_{M_\infty})\ | \ 
\bfR j_{M!!}(F)\in \bfD^{\leq 0}(\I\CC_{\che{M}})\},\\
\bfD^{\geq 0}(\I\CC_{M_\infty}) & = \{F\in \BDC(\I\CC_{M_\infty})\ | \ 
\bfR j_{M!!}(F)\in \bfD^{\geq 0}(\I\CC_{\che{M}})\}.
\end{align*}
We denote by 
\[\SH^n : \BDC(\I\CC_{M_\infty})\to\bfD^0(\I\CC_{M_\infty})\]
the $n$-th cohomology functor, where we set 
\[\bfD^0(\I\CC_{M_\infty}) :=
\bfD^{\leq 0}(\I\CC_{M_\infty})\cap\bfD^{\geq 0}(\I\CC_{M_\infty})
\simeq \I\CC_{M_\infty} := \I\CC_{\che{M}} / \I\CC_{\che{M}\backslash M}.\]

\subsection{Enhanced Ind-Sheaves}
Let us recall some basic notions and results on enhanced ind-sheaves.
References are made to D'Agnolo-Kashiwara \cite{DK16} and Kashiwara-Schapira \cite{KS16}. 
Let $M$ be a good topological space.
Set $\RR_\infty := (\RR, \var{\RR})$ for 
$\var{\RR} := \RR\sqcup\{-\infty, +\infty\}$,
and let $t\in\RR$ be the affine coordinate. 
We denote by $\Potimes, \Prihom$ the convolution functors
for ind-sheaves on $M\times\RR_\infty := (M\times \RR, M\times\var{\RR})$.
Now we define the triangulated category 
of enhanced ind-sheaves on $M$ by 
$$\BEC(\I\CC_M) := \BDC(\I\CC_{M \times\RR_\infty})/\pi^{-1}\BDC(\I\CC_M)$$
where $\pi : M\times\RR_\infty\to M$ is a morphism
of bordered spaces induced by the first projection $M\times\RR\to M$.
The quotient functor
\[\Q : \BDC(\I\CC_{M\times\RR_\infty})\to\BEC(\I\CC_M)\]
has fully faithful left and right adjoints 
$\bfL^\rmE,\bfR^\rmE : \BEC(\I\CC_M) \to\BDC(\I\CC_{M\times\RR_\infty})$ defined by 
\[\bfL^\rmE(\Q F) := (\CC_{\{t\geq0\}}\oplus\CC_{\{t\leq 0\}})
\Potimes F ,\hspace{20pt} \bfR^\rmE(\Q F) :
=\Prihom(\CC_{\{t\geq0\}}\oplus\CC_{\{t\leq 0\}}, F).\]
Moreover they induce equivalences of categories
\begin{align*}
\bfL^\rmE &: \BEC(\I\CC_M)\simto
\{F\in\BDC(\I\CC_{M\times\RR_\infty})\ |\
(\CC_{\{t\geq0\}}\oplus\CC_{\{t\leq0\}})\Potimes F\simto F\},\\
\bfR^\rmE &: \BEC(\I\CC_M)\simto
\{F\in\BDC(\I\CC_{M\times\RR_\infty})\ |\ F\simto
\Prihom(\CC_{\{t\geq0\}}\oplus\CC_{\{t\leq0\}}, F)\},
\end{align*}
respectively,
where $\{t\geq0\}$ stands for $\{(x, t)\in M\times\var{\RR}\ |\ t\in\RR, t\geq0\}$ 
and $\{t\leq0\}$ is defined similarly.
Then we have the following standard t-structure of $\BEC(\I\CC_M)$
which is induced by the standard t-structure of $\BDC(\I\CC_{M\times\RR_\infty})$:
\begin{align*}
\bfE^{\leq 0}(\I\CC_M) & = \{K\in \BEC(\I\CC_M)\ | \ 
\bfL^{\rmE}K\in \bfD^{\leq 0}(\I\CC_{M\times\RR_\infty})\},\\
\bfE^{\geq 0}(\I\CC_M) & = \{K\in \BEC(\I\CC_M)\ | \ 
\bfL^{\rmE}K\in \bfD^{\geq 0}(\I\CC_{M\times\RR_\infty})\}.
\end{align*}
We denote by 
\[\SH^n : \BEC(\I\CC_M)\to\bfE^0(\I\CC_M)\]
the $n$-th cohomology functor, where we set 
\begin{align*}
\bfE^0(\I\CC_M) &:=
\bfE^{\leq 0}(\I\CC_M)\cap\bfE^{\geq 0}(\I\CC_M)\\
&\simeq
\I\CC_{M\times\RR_\infty} / \pi^{-1}\I\CC_M\\
&\simeq
 \{F\in \I\CC_{M \times\RR_\infty}\ |\ 
(\CC_{\{t\geq0\}}\oplus\CC_{\{t\leq0\}})\Potimes F\simto F\}.
\end{align*}

The convolution functors are also defined for enhanced ind-sheaves.
We denote them by the same symbols $\Potimes$, $\Prihom$. 
For a continuous map $f : M \to N $, we 
can define also the operations 
$\bfE f^{-1}$, $\bfE f_\ast$, $\bfE f^!$, $\bfE f_{!!}$ 
for enhanced ind-sheaves. 
Moreover we have outer-hom functors
$\rihom^\rmE(K_1, K_2),  \rhom^\rmE(K_1, K_2), \rHom^\rmE(K_1, K_2)$
with values in $\BDC(\I\CC_M), 
\BDC(\CC_M)$ and $\BDC(\CC)$, respectively. 
Here, $\BDC(\CC)$ is the derived category of $\CC$-vector spaces.
For $F\in\BDC(\I\CC_M)$ and $K\in\BEC(\I\CC_M)$ the objects 
\begin{align*}
\pi^{-1}F\otimes K &:=\Q(\pi^{-1}F\otimes \bfL^\rmE K),\\
\rihom(\pi^{-1}F, K) &:=\Q\big(\rihom(\pi^{-1}F, \bfR^\rmE K)\big). 
\end{align*}
in $\BEC(\I\CC_M)$ are well defined. 
Set $\CC_M^\rmE := \Q 
\Bigl(``\underset{a\to +\infty}{\varinjlim}"\ \CC_{\{t\geq a\}}
\Bigr)\in\BEC(\I\CC_M)$. 
We say that an object $K$ of $\in\BEC(\I\CC_M)$ is stable
if $K\simeq K\Potimes\CC_M^{\rmE}$
and we denote by $\BECstb(\I\CC_M)$ the full subcategory of $\BEC(\I\CC_M)$
consisting of stable enhanced ind-sheaves on $M$.
Note that $K\in\BEC(\I\CC_M)$ is stable if and only if $K\simeq\Prihom(\CC_M^{\rmE}, K)$.
Then we have a natural embedding $e : \BDC(\CC_M) \to \BECstb(\I\CC_M)$
defined by \[e(\SF) :=  \CC_M^\rmE\otimes\pi^{-1}\SF.\]
Let $i_0 : M\to M\times\RR_\infty$ be the inclusion map of bordered spaces
induced by $x\mapsto (x, 0)$.
We set
\[\sh := \alpha_M\circ i_0^!\circ \bfR^{\rmE} : \BEC(\I\CC_M) \to \BDC(\CC_M) \]
and call it the sheafification functor.
By \cite[Lemma 3.13]{IT18},
we have $$\sh(K)\simeq \rhom^\rmE(\CC_{\{t\geq0\}}\oplus\CC_{\{t\leq 0\}}, K)$$
for an enhanced ind-sheaf $K$.
Note that there exists an isomorphism $\id\simto\sh \circ e$.

For a continuous function $\varphi : U\to \RR$ defined on an open subset $U\subset M$,
we set the exponential enhanced ind-sheaf by 
\[\EE_{U|M}^\varphi := 
\CC_M^\rmE\Potimes
\Q(\CC_{\{t+\varphi\geq0\}})
, \]
where $\{t+\varphi\geq0\}$ stands for 
$\{(x, t)\in M\times\var{\RR}\ |\ t\in\RR, x\in U, 
t+\varphi(x)\geq0\}$. 

We also define the notion of enhanced ind-sheaves on bordered space $M_\infty =(M, \che{M})$
and denote by $\BEC(\I\CC_{M_\infty})$ the triangulated category
of the enhanced ind-sheaves on $M_\infty$.
We shall skip the details of it.
Reference are made to \cite{KS16-2}.

\subsection{$\RR$-Constructible Enhanced Ind-Sheaves} 
We shall recall a notion of the $\RR$-constructability for enhanced ind-sheaves and results on it.
References are made to D'Agnolo-Kashiwara \cite{DK16-2} and \cite{DK16}.
In this subsection, we assume that $M$ is a subanalytic space.
\begin{definition}[{\cite[Definition 4.9.1]{DK16}}]
We denote by $\BDC_{\RR-c}(\CC_{M\times\RR_\infty})$
the full subcategory of $\BDC(\CC_{M\times\RR_\infty})$
consisting of objects $\SF$ satisfying
$\rmR j_{M!}\SF$ is an $\RR$-constructible sheaf on $M\times\var{\RR}$.
We regard $\BDC_{\RR-c}(\CC_{M\times\RR_\infty})$
as a full subcategory of $\BDC(\I\CC_{M\times\RR_\infty})$.
\end{definition}
\begin{definition}[{\cite[Definition 4.9.2]{DK16}}]
We say that $K\in\BEC(\I\CC_M)$ is $\RR$-constructible
if for any relatively compact subanalytic open subset $U\subset M$
there exists an isomorphism $\pi^{-1}\CC_U\otimes K\simeq \CC_M^{\rmE}\Potimes \SF$
for some $\SF\in\BDC_{\RR-c}(\CC_{M\times\RR_\infty})$.
We denote by $\BEC_{\RR-c}(\I\CC_{M})$
the full triangulated subcategory of $\BEC(\I\CC_{M})$
consisting of $\RR$-constructible enhanced ind-sheaves.
\end{definition}
$\BEC_{\RR-c}(\I\CC_M)$ has the following standard t-structure
which is induced by the standard t-structure on $\BEC(\I\CC_{M})$:
\begin{align*}
\bfE^{\leq 0}_{\RR-c}(\I\CC_M) & := \bfE^{\leq 0}(\I\CC_M)\cap \BEC_{\RR-c}(\I\CC_{M}),\\
\bfE^{\geq 0}_{\RR-c}(\I\CC_M) & := \bfE^{\geq 0}(\I\CC_M)\cap \BEC_{\RR-c}(\I\CC_{M}).
\end{align*}
We set $\ZEC_{\RR-c}(\I\CC_M) :=
\bfE^{\leq 0}_{\RR-c}(\I\CC_M)\cap\bfE^{\geq 0}_{\RR-c}(\I\CC_M)$.

The following lemma implies that
the $\RR$-constructability of enhanced ind-shaves is a local property.
\begin{lemma}[{\cite[Lemma 4.9.7]{DK16}}]\label{lem2.3}
For $K\in\BEC(\I\CC_M)$, the following conditions are equivalent $:$
\begin{itemize}
\item[\rm(i)]
$K\in\BEC_{\RR-c}(\I\CC_M),$

\item[\rm(ii)]
there exist a locally finite family $\{Z_i\}_{i\in I}$
of locally closed subanalytic subset of $M$
and $\SF_i\in\BDC_{\RR-c}(\CC_{M\times\RR_\infty})$ such that
$M=\cup_{i\in I}Z_i$ and
$$\pi^{-1}\CC_{Z_i}\otimes K\simeq \CC_M^{\rmE}\Potimes \SF_i
\mbox{ for all } i\in I.$$
\end{itemize}
\end{lemma}

We define the Verdier duality functor for enhanced ind-sheaves by
\[\rmD_M^\rmE : \BEC(\I\CC_M)^{\op}\to\BEC(\I\CC_M),\hspace{5pt}
K\mapsto \Prihom(K, \omega_M^{\rmE}),\]
where
$\omega_M^{\rmE} := e(\omega_M) = \CC_M^\rmE\otimes\pi^{-1}\omega_M$
and $\omega_M$ is the dualizing complex on $M$
(see \cite[Definition 4.8.1]{DK16}).
For any $K\in\BEC_{\RR-c}(\I\CC_M)$, $\rmD_M^{\rmE}K\in\BEC_{\RR-c}(\I\CC_M)$
and there exists an isomorphism
$K\simto \rmD_M^{\rmE}\rmD_M^{\rmE}K$ \cite[Theorem 4.9.12]{DK16}.

By \cite[Proposition 3.3.12 and Notation 3.3.13]{DK16-2},
we have a generalized t-structure
$({}_{\frac{1}{2}}\bfE_{\RR-c}^{\leq c}(\I\CC_M),
{}_{\frac{1}{2}}\bfE_{\RR-c}^{\geq c}(\I\CC_M))_{c\in\RR}$
on $\BEC_{\RR-c}(\I\CC_{M})$ defined by
\begin{align*}
{}_{\frac{1}{2}}\bfE_{\RR-c}^{\leq c}(\I\CC_M) &:=
\left\{K\in\BEC_{\RR-c}(\I\CC_{M}) \left|
\begin{array}{l}
\mbox{for any $k\in\ZZ_{\geq0}$, there exists a 
closed subanalytic subset $Z$}\\
\mbox{of dimension $<k$ with  
$\bfE i_{(M\setminus Z)_\infty}^{-1}K
\in \bfE_{\RR-c}^{\leq c-\frac{k}{2}}(\I\CC_{(M\setminus Z)_\infty})$}
\end{array}
\right.\right\},\\
{}_{\frac{1}{2}}\bfE_{\RR-c}^{\geq c}(\I\CC_M) &:=
\left\{K\in\BEC_{\RR-c}(\I\CC_{M}) \left|
\begin{array}{l}
\mbox{for any $k\in\ZZ_{\geq0}$ and any closed subanalytic subset $Z$}\\
\mbox{of dimension $\leq k$, one has
$\bfE i_{Z_\infty}^{!}K\in \bfE_{\RR-c}^{\geq c-\frac{k}{2}}(\I\CC_{Z_\infty})$}
\end{array}
\right.\right\},
\end{align*}
where $i_{Z_\infty} : Z_\infty\to M$ is a morphism of bordered spaces
given by the embedding $i_Z : Z\hookrightarrow M$.
However this t-structure does not behave well
with the duality functor $\rmD_M^{\rmE}$.
Hence we define full subcategories of $\BEC_{\RR-c}(\I\CC_{M})$ by
\begin{align*}
{}^{\frac{1}{2}}\bfE^{\leq c}_{\RR-c}(\I\CC_M)
&:=\{K\in\BEC_{\RR-c}(\I\CC_{M})\ |\ 
K\in{}_{\frac{1}{2}}\bfE^{\leq c}_{\RR-c}(\I\CC_M),
\rmD_X^{\rmE}K\in{}_{\frac{1}{2}}\bfE^{\geq -c-\frac{1}{2}}_{\RR-c}(\I\CC_M)\},\\
{}^{\frac{1}{2}}\bfE^{\geq c}_{\RR-c}(\I\CC_M)
&:=\{K\in\BEC_{\RR-c}(\I\CC_{M})\ |\
\rmD_M^{\rmE}K\in{}^{\frac{1}{2}}\bfE^{\leq -c}_{\RR-c}(\I\CC_M)\}\\
&=\{K\in\BEC_{\RR-c}(\I\CC_{M})\ |\ 
K\in{}_{\frac{1}{2}}\bfE^{\geq c-\frac{1}{2}}_{\RR-c}(\I\CC_M),
\rmD_X^{\rmE}K\in{}_{\frac{1}{2}}\bfE^{\leq -c}_{\RR-c}(\I\CC_M)\}.
\end{align*}
Then $\big({}^{\frac{1}{2}}\bfE_{\RR-c}^{\leq c}(\I\CC_M),
{}^{\frac{1}{2}}\bfE_{\RR-c}^{\geq c}(\I\CC_M)\big)_{c\in\RR}$
is a generalized t-structure of $\BEC_{\RR-c}(\I\CC_{M})$
by \cite[Theorem 3.5.2 and Definition 3.5.8]{DK16-2}
and for any $c\in\ZZ$ we have
\begin{align*}
{}^{\frac{1}{2}}\bfE^{\leq c}_{\RR-c}(\I\CC_M)
\subset
{}_{\frac{1}{2}}\bfE^{\leq c}_{\RR-c}(\I\CC_M)
&\subset
\bfE^{\leq c}_{\RR-c}(\I\CC_M),\\
\bfE^{\geq c}_{\RR-c}(\I\CC_M)
\subset
{}_{\frac{1}{2}}\bfE^{\geq c}_{\RR-c}(\I\CC_M)
&\subset
{}^{\frac{1}{2}}\bfE^{\geq c}_{\RR-c}(\I\CC_M)
\end{align*}
by \cite[Lemma 3.2.3, Lemma 3.4.4 and (3.5.1)]{DK16-2}.

\subsection{$\D$-Modules}
In this subsection we recall some basic notions and results on $\SD$-modules. 
References are made to 
\cite{HTT08}, 
\cite{Bjo93},
\cite[\S 7]{KS01}, 
\cite[\S 8, 9]{DK16} and 
\cite[\S 3, 4, 7]{KS16}.  
For a complex manifold $X$ we denote by $d_X$ its complex dimension. 
Denote by $\SO_X$ and $\SD_X$ the sheaves of holomorphic functions 
and holomorphic differential operators on $X$, respectively. 
Let $\BDC(\SD_X)$ be the bounded derived category of left $\SD_X$-modules. 
Moreover we denote by $\BDCcoh(\SD_X)$,
$\BDChol(\SD_X)$ and $\BDCrh(\SD_X)$ the full triangulated subcategories
of $\BDC(\SD_X)$ consisting of objects with coherent,
holonomic and regular holonomic cohomologies, respectively.
For a morphism $f : X\to Y$ of complex manifolds, 
denote by $\Dotimes, \rhom_{\SD_X}, \bfD f_\ast, \bfD f^\ast$, 
$\DD_X : \BDCcoh(\SD_X)^{\op} \simto \BDCcoh(\SD_X)$  
the standard operations for $\SD$-modules. 
The classical solution functor is defined by  
\begin{align*}
\Sol_X &: \BDCcoh (\SD_X)^{\op}\to\BDC(\CC_X),
\hspace{10pt}\SM \longmapsto \rhom_{\SD_X}(\SM, \SO_X).
\end{align*}
For a closed hypersurface $D$ in $X$ we denote by $\SO_X(\ast D)$ 
the sheaf of meromorphic functions on $X$ with poles in $D$. 
Then for $\SM\in\BDC(\SD_X)$ we set 
$\SM(\ast D) := \SM\Dotimes\SO_X(\ast D)$.
For $f\in\SO_X(\ast D)$ let us denote $\SE_{X\bs D|X}^f$
the meromorphic connection on $X$ along $D$ associated to $d+df$ \cite[Definition 6.1.1]{DK16}. 
Denote by $\SO_X^{\rmE}$ the enhanced ind-sheaf of tempered holomorphic functions
\cite[Definition 8.2.1]{DK16}
and by $\Sol_X^{\rmE}$ the enhanced solution functor:
\[
\Sol_X^\rmE : \BDCcoh (\SD_X)^{\op}\to\BEC(\I\CC_X), 
\hspace{10pt} 
\SM \longmapsto \rihom_{\SD_X}(\SM, \SO_X^\rmE) ,
\]
\cite[Definition 9.1.1]{DK16}.
Note that for $\SM\in\BDCcoh(\SD_X)$, 
we have an isomorphism
\[\sh\big( \Sol_X^{\rmE}(\M)\big)\simeq \Sol_X(\M)\]
by \cite[Lemma 9.5.5]{DK16}.

Let us recall the results of \cite{DK16}.
We note that (3) of Theorem \ref{thm2.5} below was proved in \cite{DK16}
under the assumption that $\M$ has a globally good filtration.
However, any holonomic $\D$-module on $X$ has a globally defined good filtration
by \cite{Mal94, Mal94-2, Mal96} (see also \cite[Theorem 4.3.4]{Sab11}).
\begin{theorem}[{\cite[\S 9.4]{DK16}}]
\label{thm2.5}

\begin{enumerate}
\item[\rm{(1)}]
 For $\SM\in\BDC_{\rm hol}(\SD_X)$ there is an isomorphism in $\BEC(\I\CC_X)$
\[\Sol_X^\rmE(\DD_X\SM)[2d_X]\simeq\rmD_X^\rmE\Sol_X^\rmE(\SM).\]

\item[\rm{(2)}] 
Let $f : X\to Y$ be a morphism of complex manifolds.
Then for $\SN\in\BDC_{\rm hol}(\SD_Y)$ there is an isomorphism in $\BEC(\I\CC_X)$
\[\Sol_X^\rmE({\bfD} f^\ast\SN)\simeq\bfE f^{-1}\Sol_Y^\rmE(\SN).\]

\item[\rm{(3)}] 
Let $f : X\to Y$ be a proper morphism of complex manifolds.
For $\SM\in\BDC_{\rm hol}(\SD_X)$ there exists an isomorphism in $\BEC(\I\CC_Y)$
\[\Sol_Y^\rmE({\bfD} f_\ast\SM)[d_Y]\simeq\bfE f_\ast \Sol_X^\rmE(\SM )[d_X].\]

\item[\rm{(4)}]
For $\SM_1, \SM_2\in\BDC_{\rm hol}(\SD_X)$,
there exists an isomorphism in $\BEC(\I\CC_X)$
\[\Sol_X^\rmE(\SM_1\Dotimes\SM_2)\simeq \Sol_X^\rmE(\SM_1)
\Potimes \Sol_X^\rmE(\SM_2).\]

\item[\rm{(5)}]
Let $\SM\in\BDC_{\rm hol}(\SD_X)$ and $D\subset X$ be a closed hypersurface,
then there exists an isomorphism in $\BEC(\I\CC_X)$
\[
\Sol_X^\rmE\big(\SM(\ast D)\big) \simeq \pi^{-1}
\CC_{X\bs D}\otimes \Sol_X^\rmE(\SM). 
\]

\item[\rm{(6)}] Let $D$ be a closed hypersurface in $X$ and 
$f\in\SO_X(\ast D)$ a meromorphic function along $D$.
Then there exists an isomorphism in $\BEC(\I\CC_X)$
\[\Sol_X^\rmE\big(\SE_{X\backslash D | X}^\varphi\big) 
\simeq \EE_{X\backslash D | X}^{\Re\varphi}.\]
\end{enumerate}
\end{theorem}

We also recall the following theorems (\cite[Theorem 9.6.1]{DK16},
 \cite[Theorem 9.1.3]{DK16} and  \cite[Theorem 4.5.1]{DK16-2}).
\begin{theorem}\label{thm2.6}
\begin{itemize}
\item[\rm(1)]
The enhanced solution functor induces an embedding
\[ \Sol_X^\rmE : \BDC_{\rm hol} (\SD_X)^{\rm op}
\hookrightarrow\BEC_{\RR-c}(\I\CC_X).\]
Moreover for any $\M\in\BDChol(\D_X)$ there exists an isomorphism
\[\M\simto \RH_X^{\rmE}\big(\Sol_X^{\rmE}(\M)\big),\]  
where $\RH_X^{\rmE}(K) := \rhom^{\rmE}(K, \SO_X^{\rmE})$.
\item[\rm(2)]
For any $\M\in\BDCrh(\D_X)$ there exists an isomorphism
$$\Sol_X^{\rmE}(\M)\simeq e\big(\Sol_X(\M)\big)$$
and hence there exists a commutative diagram
\[\xymatrix@C=30pt@M=5pt{
\BDChol(\D_X)^{\op}\ar@{^{(}->}[r]^-{\Sol_X^{\rmE}}
\ar@{}[rd]|{\rotatebox[origin=c]{180}{$\circlearrowright$}}
 & \BEC_{\RR-c}(\I\CC_X)\\
\BDCrh(\D_X)^{\op}\ar@{->}[r]_-{\Sol_X}^-{\sim}\ar@{}[u]|-{\bigcup} &\BDC_{\CC-c}(\CC_X).
\ar@{^{(}->}[u]_-{e}
}\]

\item[\rm(3)]
For any $c\in \RR$ we have
\begin{align*}
\Sol_X^{\rmE}\Big(\DChol^{\leq c}(\D_X)\Big)[d_X] &\subset
{}_{\frac{1}{2}}\bfE^{\geq -c}_{\RR-c}(\I\CC_X)\subset
{}^{\frac{1}{2}}\bfE^{\geq -c}_{\RR-c}(\I\CC_X),\\
\Sol_X^{\rmE}\Big(\DChol^{\geq c}(\D_X)\Big)[d_X] &\subset
{}^{\frac{1}{2}}\bfE^{\leq -c}_{\RR-c}(\I\CC_X),\\
\RH_X^{\rmE}\Big({}^{\frac{1}{2}}\bfE^{\leq c}_{\RR-c}(\I\CC_X)\Big)[d_X]
&\subset\bfD^{\geq -c}(\D_X).
\end{align*}
Moreover, we have
\[\Sol_X^{\rmE}\Big(\Modhol(\D_X)\Big)[d_X] \subset
{}^{\frac{1}{2}}\bfE^{\leq 0}_{\RR-c}(\I\CC_X)
\cap
{}^{\frac{1}{2}}\bfE^{\geq 0}_{\RR-c}(\I\CC_X).\]
\end{itemize}
\end{theorem}

At the end of this subsection, let us recall the notion of $\M_{\reg}$.
We denote by $\D_X^\infty$
the sheaf of rings of differential operators of infinite order on $X$, and 
by $\BDC(\D_X^\infty)$ the derived category of $\D_X^\infty$-modules.
Let us remark that $\BDC(\D_X^\infty)$ has the standard t-structure
$\big(\bfD^{\leq 0}(\D_X^\infty), \bfD^{\geq 0}(\D_X^\infty)\big)$.
We set $\M^\infty := \D_X^\infty\otimes _{\D_X}\M$ and hence
we obtain a functor
\[(\cdot)^\infty : \Mod(\D_X)\to\Mod(\D_X^\infty),
\hspace{5pt} \M\mapsto\M^\infty.\]

Note that $\D_X^\infty$ is faithfully flat over $\D_X$ \cite[p 406]{SKK}.
Hence, we also obtain a functor
$$(\cdot)^\infty : \BDC(\D_X)\to\BDC(\D_X^\infty)$$
between derived categories.
We say that a $\D_X^\infty$-module $\rm M$ is holonomic (resp.\ regular holonomic)
if there exists a holonomic (resp.\ regular holonomic) $\D_X$-modules $\M$
such that ${\rm M}\simeq \M^\infty$.
Let us denote by $\BDChol(\D_X^\infty)$ (resp.\ $\BDCrh(\D_X^\infty)$)
the full triangulated subcategory of $\BDC(\D_X^\infty)$ consisting of objects
whose cohomologies are holonomic (resp.\ regular holonomic) $\D_X^\infty$-modules.
However, by the following proposition, we have
$$\BDChol(\D_X^\infty) = \BDCrh(\D_X^\infty).$$

\begin{proposition}[{\cite[Theorem 5.5.22]{Bjo93}},
{\cite[Theorem 5.2.1]{KK}} and {\cite[Proposition 5.7]{Kas84}}]\label{prop2.7}~\\
\vspace{-20pt}
\begin{itemize}
\item[\rm (1)]
Let $\M$ be a holonomic $\D_X$-module.
Then there exists a unique regular holonomic $\D_X$-module $\M_{\reg}$
such that
\begin{itemize}
\item[\rm (i)]
$\M_{\reg}^\infty\simeq \M^\infty$,

\item[\rm (ii)]
$\M_\reg$ contains every regular holonomic $\D_X$-submodule of $\M^\infty$,

\item[\rm (iii)]
$\Sol_X(\M_\reg)\simeq \Sol_X(\M)$.
\end{itemize}

\item[\rm (2)]
There exists an isomorphism
\[\M_{\reg}\simeq
\{s\in\M^\infty\ |\ \D_X\cdot s\in\Modrh(\D_X)\}.\]
\end{itemize}
\end{proposition}
By this proposition, we obtain a functor
\[(\cdot)_{\reg} : \Modhol(\D_X)\to\Modrh(\D_X),
\hspace{5pt} \M\mapsto\M_\reg.\]
Here for a morphism $\varphi : \M\to \N$ of holonomic $\D_X$-modules,
then we set $\varphi_\reg = (\varphi^\infty)|_{\M_\reg}$.
We call it the regularization functor.
Note that this is an exact functor.
Hence, we can also consider
the functor $$(\cdot)_{\reg} : \BDChol(\D_X)\to\BDCrh(\D_X)$$
between derived categories.

\subsection{$\D^{\SA}$-Modules}
In this subsection we recall some notions and results
on $\SD_{\tl{X}}^\SA$ in \cite[\S 7]{DK16}. 
Let $X$ be an $n$-dimensional complex manifold
and $Y \subset X$ a smooth closed hyepersurface.
The real blow-up $\varpi_X : \tl{X}_Y\to X$ of $X$ along $Y$ is the real analytic map of
real analytic manifolds locally defined as follows.
We take local coordinates $(z, w)\in\CC\times\CC^{n-1}$ on $X$
such that $Y=\{z=0\}$.
Then we has 
\[\tl{X}_Y = \{(t, \zeta, w)\in\RR\times\CC\times\CC^{n-1}\ |\ 
|\zeta|=1, t\geq0 \}\]
and
\[\varpi_X : \tl{X}_Y\to X\hspace{10pt} (t, \zeta, w)\mapsto (t\zeta, w).\]
Let now $D$ be a normal crossing divisor of $X$,
and locally write $D=D_1\cup\cdots\cup D_r$
where $D_i$ is a smooth closed hypersurface of $X$.
Then the real blow-up of $X$ along $D$ is defined by
\[
\tl{X}_D := \tl{X}_{D_1}\underset{X}{\times}\cdots\underset{X}{\times}\tl{X}_{D_r}
\]
and
\[
\varpi_X : \tl{X}_D\to X.
\]
Sometimes we abbreviate $\varpi_X$ to $\varpi : \tl{X}\to X$ for simplicity. 
Denote by $\SO_{\tl{X}}^{\rmt}\in\BDC(\I\CC_{\tl{X}})$
the ind-sheaf of tempered holomorphic functions on $\tl{X}$.
See \cite[Notation 7.2.6]{DK16} for the definition.
We set $\SA_{\tl{X}} := \alpha_{\tl{X}}(\SO_{\tl{X}}^{\rmt})$
(see \cite[Proposition 7.2.10]{DK16} for precisely)
and 
\begin{align*}
\SD_{\tl{X}}^\SA &:= \SA_{\tl{X}}\otimes_{\varpi^{-1}\SO_X}
\varpi^{-1}\SD_X,\\
\SM^{\SA} &:= \SD_{\tl{X}}^\SA\Lotimes
{\varpi^{-1}\SD_X}\varpi^{-1}\SM
\simeq\SA_{\tl{X}}
\Lotimes{\varpi^{-1}\SO_X}\varpi^{-1}\SM
\end{align*}
for $\M\in\BDC(\SD_X)$.
Recall that a section of $\SA_{\tl X}$ is a holomorphic function having
moderate growth at $\varpi_X^{-1}(D)$.
Note that $\SA_{\tl{X}}$ and 
$\SD_{\tl{X}}^\SA$ are sheaves of rings on $\tl{X}$. 
For $\mathscr{M}\in\BDC(\SD_{\tl{X}}^\SA)$
we define the enhanced solution functor on $\tl{X}$ by 
\[
\Sol_{\tl{X}}^\rmE(\mathscr{M}) := 
\rihom_{\SD_{\tl{X}}^\SA}(\mathscr{M}, \SO_{\tl{X}}^\rmE)
\]
where $\SO_{\tl{X}}^{\rmE}\in\BEC(\I\CC_{\tl{X}})$ is
the enhanced ind-sheaf of tempered holomorphic functions on $\tl{X}$
(See \cite[Definition 9.2.1]{DK16} for the definition).

From now on, we introduce the result of K.S. Kedlaya and T. Mochizuki.
Let $X$ be a complex manifold and 
$D \subset X$ a normal crossing divisor in it. 
Let us take local coordinates 
$(u_1, \ldots, u_l, v_1, \ldots, v_{d_X-l})$ 
of $X$ such that $D= \{ u_1 u_2 \cdots u_l=0 \}$. 
We define a partial order $\leq$ on the 
set $\ZZ^l$ by 
\[ a \leq a^{\prime} \ \Longleftrightarrow 
\ a_i \leq a_i^{\prime} \ (1 \leq i \leq l).\] 
Then for a meromorphic function $\varphi\in\SO_X(\ast D)$
on $X$ along $D$ which has the Laurent expansion
\[ \varphi = \sum_{a \in \ZZ^l} c_a( \varphi )(v) \cdot 
u^a \ \in \SO_X(\ast D) \]
with respect to $u_1, \ldots, u_{l}$,
we define its order 
$\ord( \varphi ) \in \ZZ^l$ by the minimum 
\[ \min \Big( \{ a \in \ZZ^l \ | \
c_a( \varphi ) \not= 0 \} \cup \{ 0 \} \Big) \]
if it exists. 
For any $f\in\SO_X(\ast D)/ \SO_X$, we take any lift $\tl{f}$ to $\SO_X(\ast D)$,
and we set $\ord(f) := \ord(\tl{f})$, if the right hand side exists.
Note that it is independent of the choice of a lift $\tl{f}$.
If $\ord(f)\neq0$, $c_{\ord(f)}(\tl{f})$ is independent of the choice of a lift $\tl{f}$,
which is denoted by $c_{\ord(f)}(f)$.
\begin{definition}[{\cite[Definition 2.1.2]{Mochi11}}]\label{def2.10}
In the situation as above,
let us set $$Y= \{ u_1=u_2= \cdots =u_l=0 \}.$$
A finite subset $\calI \subset \SO_X(\ast D)/ \SO_X$
is called a good set of irregular values on $(X,D)$,
if the following conditions are satisfied:
\begin{itemize}
\setlength{\itemsep}{-3pt}
\item[-]
$\ord(f)$ exists for each element $f\in\calI$.
If $f\neq0$ in $\SO_X(\ast D)/ \SO_X$, $c_{\ord(f)}(f)$ is invertible on $Y$.
\item[-]
$\ord(f-g)$ exists for two distinct $f, g\in\calI$,
$c_{\ord(f-g)}(f-g)$ is invertible on $Y$.
\item[-]
The set $\{\ord(f-g)\ |\ f, g\in\calI\}$ is totally ordered
with respect to the above partial order $\leq$ on $\ZZ^l$.
\end{itemize}
\end{definition}

\begin{definition}\label{def2.8}
We say that a holonomic $\SD_X$-module $\SM$ has a normal form along $D$ if
\begin{itemize}
\setlength{\itemsep}{-3pt}
\item[(i)]
$\SM\simto\SM(\ast D)$
\item[(ii)] $\singsupp(\SM)\subset D$
\item[(iii)] for any $\theta\in\varpi^{-1}(D)\subset\tl{X}$,
there exist an open neighborhood $U\subset X$
of $\varpi({\theta})$,
a good set of irregular values $\{\varphi_i\}$ on $(U, U\cap D)$
and an open neighborhood $V$ of $\theta$ 
with $V\subset\varpi^{-1}(U)$
such that
\[
(\SM|_U)^\SA|_V
\simeq
\Bigl(
\bigoplus_i\bigl(\SE_{U\bs D|U}^{\varphi_i}\bigr)^\SA
\Bigr)
|_V.\]
\end{itemize}
\end{definition}

A ramification of $X$ along a normal crossing divisor $D$ on a neighborhood $U$ 
of $x \in D$ is a finite map $p : U'\to U$ of complex manifolds of the form
$z' \mapsto z=(z_1,z_2, \ldots, z_n)= 
 p(z') = (z'^{m_1}_1,\ldots, z'^{m_r}_r, z'_{r+1},\ldots,z'_n)$ 
for some $(m_1, \ldots, m_r)\in (\ZZ_{>0})^r$, where 
$(z'_1,\ldots, z'_n)$ is a local coordinate system of $U'$ and 
$(z_1, \ldots, z_n)$ is the one of 
$U$ such that $D \cap U=\{z_1\cdots z_r=0\}$. 

\begin{definition}\label{def2.9}
We say that a holonomic $\SD_X$-module $\SM$ has a quasi-normal form along $D$
if it satisfies the conditions (i), (ii) in Definition \ref{def2.8} and if for any $x \in D$ 
there exists a ramification $p_x : U_x'\to U_x$ 
on a neighborhood $U_x$ of $x$ such that $\bfD p_x^\ast(\SM|_{U_x})$
has a normal form along $p_x^{-1}(D\cap U_x)$.
\end{definition}

Note that $\bfD p_x^\ast(\SM|_{U_x})$ as well 
as $\bfD p_{x\ast}\bfD p_x^\ast(\SM|_{U_x})$
is concentrated in degree zero and $\SM|_{U_x}$ is a 
direct summand of $\bfD p_{x\ast}\bfD p_x^\ast(\SM|_{U_x})$. 

A modification of $X$ with respect to an analytic hypersurface $Y$
is a projective map $f : X'\to X$ such that
$D' := f^{-1}(Y)$ is a normal crossing divisor of $X'$
and $f$ induces an isomorphism $X'\setminus D'\simto X\setminus Y$.
The following fundamental result is due to
K.S. Kedlaya and T. Mochizuki:

\begin{theorem}[\cite{Ked10, Ked11, Mochi09, Mochi11}]
For a holonomic $\SD_X$-module $\SM$ and $x\in X$,
there exist an open neighborhood $U_x$ of $x$, 
a closed hypersurface $Y_x\subset U_x$ and
a modification $f_x : U'_x\to U_x$ with respect to $Y_x$ such that
\begin{itemize}
\setlength{\itemsep}{-3pt}
\item[\rm(i)]
$ \singsupp (\SM)\cap U_x\subset Y_x$,
\item[\rm(ii)] $(\bfD f_x^\ast\SM)(\ast D_x')$ has a quasi-normal form along $D_x$,
where $D_x':=f_x^{-1}(Y_x)$ is a normal crossing divisor of $U'_x$.
\end{itemize}
\end{theorem}

\begin{corollary}\label{cor2.11}
Let $\M$ be a meromorphic connection on $X$ along an analytic hypersurface $Y$.
Then for any $x\in Y$ there exist an open neighborhood $U_x$
and a modification $f_x : U_x'\to U_x$ along $Y_x := Y\cap U_x$
such that $\bfD f_x^{\ast}\M$ has a quasi-normal form along $D_x' := f_x^{-1}(Y_x)$.
\end{corollary}

At the end of this subsection,
we shall introduce the following results of a joint work with K. Takeuchi \cite{IT18}:

\begin{theorem}[{\cite[Theorem 3.12]{IT18}}]\label{thm2.12}
Let $X$ be a complex manifold and $D$ a normal crossing divisor in it.
For $\SM\in\BDC_{\rm hol}(\SD_X)$ and an 
open sector $V\subset X\setminus D$ along $D$
we set $K:=\pi^{-1}\CC_V\otimes \Sol_X^{\rmE}(\SM)$.
Then for any open subset $W$ of $\tl{X}$ such that
$W\cap\varpi^{-1}(D)\neq\emptyset, \var{W}\subset \Int\Big(
\var{\varpi^{-1}(V)}\Big)$,
there exists  an isomorphism
\[\SM^\SA|_W\simeq\rhom^\rmE\Big(\big({\bfE}\varpi^!
\rihom(\pi^{-1}\CC_{X\bs D}, K)\big)|_W, \SO_{\tl{X}}^\rmE|_W\Big)\]
in $\BDC(\SD_{\tl{X}}^\SA)$.
\end{theorem}
This result means that
we can reconstruct the $\SD_{\tl{X}}^\SA$-module structure
of $\SM^\SA$ on $W\subset \tl{X}$ by the enhanced ind-sheaf 
$K = \pi^{-1}\CC_V \otimes \Sol_X^{\rmE}(\SM )$. 
We regard Theorem \ref{thm2.12}
as a sectorial refinement of the irregular Riemann-Hilbert 
correspondence of \cite{DK16}.
Conversely,
we can reconstruct the enhanced ind-sheaf 
$\pi^{-1}\CC_V \otimes \Sol_X^{\rmE}(\SM )$ by
the $\SD_{\tl{X}}^\SA$-module structure
of $\SM^\SA$ on $W\subset \tl{X}$
as follows:
\begin{theorem}[{\cite[Theorem 3.8]{IT18}}]\label{thm2.13}
Let $X$ be a complex manifold and $D$ a normal crossing divisor in it.
For $\SM\in\BDC_{\rm hol}(\SD_X)$ and
an open subset $W$ of $\tl{X}$ such that
$W\cap\varpi^{-1}(D)\neq\emptyset$,
we set $\mathscr{K} := {\bf E}i_W^{-1}\Sol_{\tl X}^{\rmE}(\M^\SA) 
= \Sol_W^{\rmE}(\M^\SA|_W)$,
where $i_W : W\xhookrightarrow{\ \ \ } \tl X$ is the inclusion map.
Then for any open sector $V\subset X\setminus D$ along $D$ 
such that $\tl{V} := \var{\varpi^{-1}(V)}\subset W$, 
there exists an isomorphism
\[\pi^{-1}\CC_V\otimes \Sol_X^{\rmE}(\M)\simeq
{\bf E}\varpi_{\ast}(\pi^{-1}\CC_{\varpi^{-1}(V)}\otimes
{\bf E}i_{\tl{V}\ast}{\bf E} j^{-1}\mathscr{K})\]
in $\BEC(\I\CC_X)$, where
$j : {\tl V}\xhookrightarrow{\ \ \ } W$ and 
$i_{\tl V} : \tl V \xhookrightarrow{\ \ \ } {\tl X}$
are the inclusion maps.
\end{theorem}

The following proposition will be used in the proof of Sublemma \ref{sublem3.4}:
\begin{proposition}[{\cite[Proposition 3.19]{IT18}}]\label{prop2.14}
In the situation as above,
let $\varpi_X : \tl X\to X$ be the 
real blow-up of $X$ along the normal 
crossing divisor $D$. Assume that 
$\varphi_1, \ldots, \varphi_m$ 
$($resp.\ $\psi_1, \ldots, \psi_m)$ 
$\in \SO_X(\ast D)/ \SO_X$ form 
a good set of irregular values on $(X,D)$. 
Assume also that for a point $\theta \in  
\varpi^{-1}_X( Y ) \subset \varpi^{-1}_X( D )$ 
there exists its open neighborhood $U$ in $\tl X$ 
on which we have an isomorphism
\[\Phi : \bigoplus_{j=1}^m \SA_{\tl X} e^{\varphi_j}\simto
\bigoplus_{i=1}^m \SA_{\tl X} e^{\psi_i}\]
of $\D_{\tl X}^\SA$-modules,
where $Y$ is the subset of $X$ in Definition \ref{def2.10}. 
Then after reordering $\varphi_j$'s and $\psi_i$'s 
for any $1\leq j\leq m$
we have $\varphi_j = \psi_j$.
\end{proposition} 

\section{Main Result}
In this section,
we define $\CC$-constructible enhanced ind-sheaves and prove that they are
nothing but the images of objects of $\BDChol(\D_X)$ via the enhanced solution functor.

\subsection{Ind-Stalk for Enhanced Ind-Sheaves}
In this subsection,
we define ind-stalks for an enhanced ind-sheaf.
Let $M$ be a good topological space
(i.e., a locally compact Hausdorff space
which is countable at infinity and has finite soft dimension).

\begin{definition}
Let $Z$ be a locally closed subset of $M$.
For $K\in\BEC(\I\CC_M)$,
we set $${}_ZK := K\otimes \pi^{-1}(\beta_M\CC_Z) \in\BEC(\I\CC_M).$$

Sometimes,  we abbreviate ${}_{\{x\}}K$ to ${}_xK$ for $x\in M$
and we call ${}_xK$ the ind-stalk of $K$ at $x\in M$.
\end{definition}
Remark that the functor ${}_Z(\cdot) : \BEC(\I\CC_M)\to\BEC(\I\CC_M)$
is t-exact with respect to the standard t-structure
by \cite[Lemma 2.7.5 (i)]{DK16-2}.

\begin{proposition}\label{prop3.2}
Let $K, L\in\BEC(\I\CC_M)$ and $\Phi :K\to L$ be a morphism of enhanced ind-sheaves. 
If the morphism $${}_x\Phi : {}_xK\to {}_xL$$ induced by $\Phi$
is an isomorphism for any $x\in M$, 
then $\Phi$ is an isomorphism.
\end{proposition}
\begin{proof}
Since there exists an object $M_\Phi\in\BEC(\I\CC_M)$ such that 
a triangle
$$K\xrightarrow{\ \Phi\ }L\xrightarrow{\ ~ \ }M_\Phi\xrightarrow{\ +1\ }$$
is a distinguished triangle,
it is enough to show that
if ${}_x(M_\Phi)\simeq0$ for any $x\in M$
then $M_\Phi\simeq0$ in $\BEC(\I\CC_M)$. 
This follows from Lemma \ref{lem}.
\end{proof}

\begin{lemma}\label{lem}
Let $K$ be an object of $\BEC(\I\CC_M)$.
If ${}_xK\simeq0$ in $\BEC(\I\CC_M)$ for any $x\in M$,
then we have $K\simeq0$ in $\BEC(\I\CC_M)$.
\end{lemma}

\begin{proof}
Let $K = \Q(F)$.
Namely $K$ is represented by $F\in\BDC(\I\CC_{M\times\RR_\infty})$.
Then we have ${}_xK = \Q(F\otimes\pi^{-1}(\beta_M\CC_x))$.
By the assumption,
there exist isomorphisms
\begin{align*}
0 &\simeq \bfL^\rmE\big(\Q\big(F\otimes\pi^{-1}(\beta_M\CC_x)\big)\big)\\
&\simeq
(\CC_{\{t\leq0\}}\oplus\CC_{\{t\geq0\}})\Potimes
\big(F\otimes\pi^{-1}(\beta_M\CC_x)\big)\\
&\simeq
\big(
(\CC_{\{t\leq0\}}\oplus\CC_{\{t\geq0\}})\Potimes F\big)\otimes\pi^{-1}(\beta_M\CC_x)
\end{align*}
in $\BDC(\I\CC_{M\times\RR_\infty})$ for any $x\in M$,
where in the last isomorphism we used \cite[Lemma 4.3.1]{DK16}.
Therefore by Sublemma \ref{sublem} below,
we have $$\bfL^\rmE(K) \simeq
(\CC_{\{t\leq0\}}\oplus\CC_{\{t\geq0\}})\Potimes F\simeq0$$
in $\BDC(\I\CC_{M\times\RR_\infty})$ and hence $K\simeq0$ in $\BEC(\I\CC_M)$.
\end{proof}

\begin{sublemma}\label{sublem}
Let $F\in\BDC(\I\CC_{M\times\RR_\infty})$.
If $F\otimes\pi^{-1}(\beta_M\CC_x) \simeq 0$
in $\BDC(\I\CC_{M\times\RR_\infty})$ for any $x\in M$,
we have $F\simeq0$ in $\BDC(\I\CC_{M\times\RR_\infty})$.
\end{sublemma}
\begin{proof}
Let $F=\q(\SF)$.
Namely, $F$ is represented by $\SF\in\BDC(\I\CC_{M\times\var{\RR}})$.
In this case by \cite[Lemma 3.3.12]{DK16} we have 
$$F\otimes\pi^{-1}(\beta_M\CC_x)\simeq
\q\big(\SF\otimes\var{\pi}^{-1}(\beta_M\CC_x)\big)$$
where $\var{\pi} : M\times\var{\RR}\to M$ is the canonical projection.
By the assumption, we have isomorphisms
\begin{align*}
0 &\simeq \bfl\big(\q\big(\SF\otimes\var{\pi}^{-1}(\beta_M\CC_x)\big)\big)\\
&\simeq
\CC_{M\times\RR}\otimes\big(\SF\otimes\var{\pi}^{-1}(\beta_M\CC_x)\big)\\
&\simeq
\CC_{M\times\RR}\otimes
\big(\SF\otimes\beta_{M\times\var{\RR}}\CC_{\{x\}\times\var{\RR}}\big)\\
&\simeq
\big(\CC_{M\times\RR}\otimes\SF\big)
\otimes\beta_{M\times\var{\RR}}\CC_{\{x\}\times\var{\RR}}\\
&\simeq
{}_{\{x\}\times\var{\RR}}(\CC_{M\times\RR}\otimes\SF)
\end{align*}
 in $\BDC(\I\CC_{M\times\var{\RR}})$ for any $x\in M$,
where in the last isomorphism we used \cite[Proposition 4.2.14 (i)]{KS01}.
Therefore we have $$\bfl(F)\simeq\CC_{M\times\RR}\otimes\SF\simeq0$$
in $\BDC(\I\CC_{M\times\var{\RR}})$ by \cite[Proposition 4.3.21]{KS01}
and hence $F\simeq0$ in $\BDC(\I\CC_{M\times\RR_\infty})$.
\end{proof}

\begin{remark}\label{rem-stalk}
Let $U\subset M$ be an open subset of $M$ and $i_U : U\hookrightarrow M$ the open embedding.
Then we have isomorphisms in $\I\CC_M$
$$\beta_M\CC_U\simeq \bfR i_{U!!}i_U^{-1}(\beta_M\CC_M)
\simeq \bfR i_{U!!}i_U^{-1}(\iota_M\CC_M)$$
by \cite[Proposition 4.3.17, Corollary 4.3.7, Example 3.3.25 and Theorem 3.3.26]{KS01}.
Hence for $K\in\BEC(\I\CC_M)$ there exists an isomorphism in $\BEC(\I\CC_M)$
$${}_UK \simeq \bfE i_{U!!}\bfE i_U^{-1}K.$$
\end{remark}

\subsection{Normal Form}
In this subsection,
we define enhanced ind-sheaves
which have a normal form along a normal crossing divisor
and prove that they are nothing but the images of holonomic $\D$-modules
which have a normal form via the enhanced solution functor.
Let $X$ be a complex manifold and $D$ a normal crossing divisor of $X$.

\begin{definition}\label{def3.1}
We say that an $\RR$-constructible enhanced ind-sheaf
$K\in\ZEC_{\RR-c}(\I\CC_X)$ has a normal form along $D$ if 
\begin{itemize}
\setlength{\itemsep}{-3pt}
\item[(i)]
$\pi^{-1}\CC_{X\setminus D}\otimes K\simto K$,

\item[(ii)]
for any $x\in X\setminus D$ there exist an open neighborhood $U_x\subset X\setminus D$
of $x$ and a non-negative integer $k$ such that
\[K|_{U_x}\simeq (\CC_{U_x}^{\rmE})^{\oplus k},\]

\item[(iii)]
for any $x\in D$ there exist an open neighborhood $U_x\subset X$ of $x$,
a good set of irregular values $\{\varphi_i\}_i$ on $(U_x, D\cap U_x)$
and a finite sectorial open covering $\{U_{x, j}\}_j$ of $U_x\bs D$
such that
\[\pi^{-1}\CC_{U_{x, j}}\otimes K|_{U_x}\simeq
\bigoplus_i \EE_{U_{x, j} | U_x}^{\Re\varphi_i}
\hspace{10pt} \mbox{for any } j.\]
\end{itemize}
\end{definition}

The following sublemma will be used later in this paper.
We shall skip the details of enhanced ind-sheaves with ring actions.
References are made to \cite[\S 5.4, 5.5 and 5.6]{KS01}, \cite[\S 4.10]{DK16},
\cite[\S 2.7]{KS16-2} and \cite[\S 6.7]{KS16}.
\begin{sublemma}\label{sublem3.3}
Let $M$ be a subanalytic space and $\SA$ a sheaf of $\CC$-algebras on $M$
which has a finite flat dimension.
Let $K\in\BEC_{\RR-c}(\I\CC_M)$, $L\in\BECstb(\I\CC_M)\cap\BEC(\I\SA)$ 
and $\SF\in\BDC(\SA^{\op})$.
Then we have an isomorphism
$$\SF\Lotimes{\SA}\rhom^\rmE(K, L)\simto
\rhom^\rmE(K, \pi^{-1}\beta_M\SF\Lotimes{\pi^{-1}\beta_M\SA}L).$$
\end{sublemma}

\begin{proof}
Note that there exists a canonical morphism
$$\SF\Lotimes{\SA}\rhom^\rmE(K, L)\to
\rhom^\rmE(K, \pi^{-1}\beta_M\SF\Lotimes{\pi^{-1}\beta_M\SA}L).$$
We shall prove that it is an isomorphism.
Since $K$ is $\RR$-constructible, 
we may assume $K=\CC_M^\rmE\Potimes \SG$
for $\SG\in\BDC(\CC_{M\times\RR_\infty})$. 
Then we have isomorphisms
\begin{align*}
\rhom^\rmE(K, L) &=
\rhom^\rmE(\CC_M^\rmE\Potimes \SG, L)\\
&\simeq
\rhom^\rmE\big(\SG, \Prihom(\CC_M^\rmE, L)\big)\\
&\simeq
\rhom^\rmE(\SG, L),
\end{align*}
where in the last isomorphism we used the assumption
that $L$ is a stable object.
Note that there exists an isomorphism
$$
\rhom^\rmE(\SG, L)\simeq
\alpha_M\bfR\var{\pi}_\ast\rihom(\bfR j_{M!!}\bfL^\rmE\SG,
\bfR j_{M!!}\bfL^\rmE L)
$$
by \cite[Lemma 3.3.7 (iv), Lemma 4.5.12]{DK16},
where $\var{\pi} : M\times\var{\RR}\to M$ is the first projection.
Hence we have isomorphisms
\begin{align*}
\SF\Lotimes{\SA}\rhom^\rmE(K, L)
&\simeq
\SF\otimes_{\SA}\rhom^\rmE(\SG, L)\\
&\simeq\SF\Lotimes{\SA}
\big(\alpha_M\bfR\var{\pi}_\ast\rihom(\bfR j_{M!!}\bfL^\rmE\SG,
\bfR j_{M!!}\bfL^\rmE L)\big)\\
&\simeq
\alpha_M\bfR\var{\pi}_\ast\big(\var{\pi}^{-1}\beta_M\SF
\Lotimes{\var{\pi}^{-1}\beta_M\SA}\rihom(
\bfR j_{M!!}\bfL^\rmE\SG, \bfR j_{M!!}\bfL^\rmE L)\big)\\
&\simeq
\alpha_M\bfR\var{\pi}_\ast
\rihom(\bfR j_{M!!}\bfL^\rmE\SG, \var{\pi}^{-1}\beta_M\SF
\Lotimes{\var{\pi}^{-1}\beta_M\SA}\bfR j_{M!!}\bfL^\rmE L )\\
&\simeq
\alpha_M\bfR\var{\pi}_\ast
\rihom(\bfR j_{M!!}\bfL^\rmE\SG, \bfR j_{M!!}\big(\pi^{-1}\beta_M\SF
\Lotimes{\pi^{-1}\beta_M\SA}\bfL^\rmE L\big)\big)\\
&\simeq
\alpha_M\bfR\var{\pi}_\ast
\rihom(\bfR j_{M!!}\bfL^\rmE\SG, \bfR j_{M!!}\bfL^\rmE\big(\pi^{-1}\beta_M\SF
\Lotimes{\pi^{-1}\beta_M\SA} L\big)\big)\\
&\simeq
\rhom^\rmE(\SG, \pi^{-1}\beta_M\SF\Lotimes{\pi^{-1}\beta_M\SA}L),
\end{align*}
in the thirwr.\ forth, fifth, sixth) isomorphism
we used \cite[Theorem 5.2.7]{KS01}
(resp.\ \cite[Theorem 5.6.1 (ii)]{KS01},
\cite[Lemma 3.3.7 (iv)]{DK16},
\cite[Lemma 4.3.1]{DK16}).
More over, since $L$ is a stable object,
$\pi^{-1}\beta_M\SF\otimes_{\pi^{-1}\beta_M\SA}L$ is also stable
by \cite[Lemma 4.3.1]{DK16}.
Then we have
\[\rhom^\rmE(\SG, \pi^{-1}\beta_M\SF\Lotimes{\pi^{-1}\beta_M\SA}L)
\simeq
\rhom^\rmE(K, \pi^{-1}\beta_M\SF\Lotimes{\pi^{-1}\beta_M\SA}L),
\]
and hence the proof is completed.
\end{proof}

We need the following sublemmas to prove Lemma \ref{lem3.3} below:
\begin{sublemma}\label{sublem3.2}
Let $Y$ be an analytic hypersurface of $X$.
\begin{itemize}
\item[\rm(1)]
If a holonomic $\D_X$-module $\M$ satisfies
\begin{itemize}
\setlength{\itemsep}{-3pt}
\item[$(a)$]
$\SM\simto\SM(\ast Y)$,
\item[$(b)$] 
$\singsupp(\SM)\subset Y$
\end{itemize} 
then the enhanced solution complex $K := \Sol_X^{\rmE}(\M)\in\ZEC_{\RR-c}(\I\CC_X)$ of $\M$ satisfies
\begin{itemize}
\setlength{\itemsep}{-3pt}
\item[$(a)'$]
$\pi^{-1}\CC_{X\setminus Y}\otimes K\simto K$,

\item[$(b)'$]
for any $x\in X\setminus Y$ there exist an open neighborhood $U_x\subset X\setminus Y$
of $x$ and a non-negative integer $k$ such that
\[K|_{U_x}\simeq (\CC_{U_x}^{\rmE})^{\oplus k}.\]
\end{itemize} 

\item[\rm(2)]
On the other hand, 
If $K\in\ZEC_{\RR-c}(\I\CC_X)$ satisfies the above conditions $(a)'$ and $(b)'$,
then $\M = \RH_X^{\rmE}(K)\in\BDC(\D_X)$ satisfies the above conditions $(a)$ and $(b)$.
\end{itemize}
\end{sublemma}

\begin{proof}
\item[(1)]
By Theorem \ref{thm2.5} (5) and the condition $(a)$,
we have isomorphisms
\begin{align*}
\pi^{-1}\CC_{X\bs Y}\otimes K &= \pi^{-1}\CC_{X\bs Y}\otimes \Sol_X^{\rmE}(\M)\\
&\simeq\Sol_X^{\rmE}\big(\M(\ast Y)\big)\\
&\simeq \Sol_X^{\rmE}(\M) = K.
\end{align*}
By the definition of $\singsupp(\M)\subset Y$, 
for any $x\in X\setminus Y$ there exist an open neighborhood $U_x\subset X\setminus Y$
of $x$ and a non-negative integer $k$ such that $\M|_{U_x}\simeq \SO_{U_x}^{\oplus k}$
and hence $$K|_{U_x} = \Sol_X^{\rmE}(\M)|_{U_x}\simeq
\Sol_{U_x}^{\rmE}(\M|_{U_x})\simeq\Sol_{U_x}^{\rmE}(\SO_{U_x}^{\oplus k})
\simeq(\CC_{U_x}^{\rmE})^{\oplus k},$$
by the fact that
there exists an isomorphism $\Sol_X^{\rmE}(\SO_X)\simeq \CC_X^{\rmE}$
(see Theorem \ref{thm2.6} (2)).
Therefore, $K$ satisfies the condition $(b)'$ as above.
Remark that $K := \Sol_X^\rmE(\M)\in\BEC(\I\CC_X)$ is concentrated in degree zero
by \cite[Corollary 5.21, Lemma 9.5, Proposition 9.6 and Theorem 9.3]{Mochi16}.

\item[(2)]
First we shall show that the condition $(a)'$ implies that
$\M = \RH_X^{\rmE}(K)$ satisfies the condition $(a)$.
In fact, we have isomorphisms
\begin{align*}
\M(\ast Y) & = \M\Dotimes\SO_X(\ast Y)\\
&=\rhom^{\rmE}(K, \SO_X^{\rmE})\Dotimes\SO_X(\ast Y)\\
&\simeq\rhom^{\rmE}\big(K, \SO_X^{\rmE}\Dotimes\SO_X(\ast Y)\big)\\
&\simeq\rhom^{\rmE}\big(K, \rihom(\pi^{-1}\CC_{X\bs Y}, \SO_X^{\rmE})\big)\\
&\simeq\rhom^{\rmE}(\pi^{-1}\CC_{X\bs Y}\otimes K, \SO_X^{\rmE})\\
&\simeq\rhom^{\rmE}(K, \SO_X^{\rmE}) = \M,
\end{align*}
where the third isomorphism follows from Sublemma \ref{sublem3.3},
the forth one follows from
$\SO_X^{\rmE}\Dotimes\SO_X(\ast Y)\simeq\rihom(\pi^{-1}\CC_{X\bs Y}, \SO_X^{\rmE})$ (see, e.g., \cite[p.88]{KS16}),
and the sixth one follows from the condition $(a)'$.

We shall show that the condition $(b)'$ implies that 
$\M$ satisfies the condition $(b)$.
Now, for any $x\in X\setminus D$
there exist an open neighborhood $U_x\subset X\setminus D$
of $x$ and a non-negative integer $k$
such that $K|_{U_x}\simeq (\CC_{U_x}^{\rmE})^{\oplus k}.$
Hence, we have isomorphisms
\begin{align*}
\M|_{U_x} &\simeq \rhom^{\rmE}(K, \SO_X^{\rmE})|_{U_x}\\
&\simeq \rhom^{\rmE}(K|_{U_x}, \SO_{U_x}^{\rmE})\\
&\simeq \rhom^{\rmE}\big((\CC_{U_x}^{\rmE})^{\oplus k}, \SO_{U_x}^{\rmE}\big)\\
&\simeq \rhom^{\rmE}(\CC_{U_x}^{\rmE}, \SO_{U_x}^{\rmE})^{\oplus k}
\simeq \SO_{U_x}^{\oplus k},
\end{align*}
where in the last isomorphism
we used the fact $\RH_X^{\rmE}(\CC_X^{\rmE})
\simeq\RH_X^{\rmE}\big(\Sol_X^{\rmE}(\SO_X)\big)\simeq\SO_X$
(see, Theorem \ref{thm2.6} (1)).
\end{proof}

\begin{sublemma}\label{sublem3.4}
The condition $\rm (iii)$ in Definition \ref{def2.8}
is equivalent to the following condition ${\rm (iii)' : }$
for any $x\in D$, 
there exist an open neighborhood $V_x\subset X$ of $x$,
a good set of irregular values $\{\varphi_i^x\}$ on $(V_x, V_x\cap D)$ such that
for any $\theta\in\varpi^{-1}(x)$
there exists an an open neighborhood $W_{x, \theta}$ of $\theta$ such that 
\[
W_{x, \theta}\subset\varpi^{-1}(V_x) \mbox{ and }
(\SM|_{V_x})^\SA|_{W_{x, \theta}}
\simeq \Bigl(
\bigoplus_i\bigl(\SE_{V_x\bs D | V_x}^{\varphi_i^x}\bigr)^\SA\Bigr)|_{W_{x, \theta}}.\]
\end{sublemma}

\begin{proof}
We shall only prove $\rm(iii) \Rightarrow (iii)'$.
Let $\M$ be a holonomic $\D_X$-module which has a normal form along $D$.
Then, 
for any $x\in D$ and any $\theta\in\varpi^{-1}(x)\subset \varpi^{-1}(D)$
there exist an open neighborhood $V_{x, \theta}$ of $x=\varpi(\theta)$, 
a good set of irregular values $\{\varphi_i^{x, \theta}\}_i$
on $(V_{x, \theta}, V_{x, \theta}\cap D)$
and an open subset $W_{x, \theta}\subset\varpi^{-1}(V_{x, \theta})$
such that 
\[(\M|_{V_{x, \theta}})^\SA|_{W_{x, \theta}}\simeq\bigoplus_i (
\E_{V_{x, \theta}\bs D | V_{x, \theta}}^{\varphi_i^{x, \theta}})^\SA|_{W_{x, \theta}}.\]
Let us consider two points $\theta, \eta\in\varpi^{-1}(x)$
such that $W_{x, \theta}\cap W_{x, \eta}\neq\emptyset$.
Then we obtain an isomorphism
\[\bigoplus_i \big(\E_{V_{x, \theta}\cap V_{x, \eta}\bs D | V_{x, \theta}\cap V_{x, \eta}}
^{\varphi_i^{x, \theta}|_{V_{x, \theta}\cap V_{x, \eta}}}
\big)^\SA|_{W_{x, \theta}\cap W_{x, \eta}}
\simeq
\bigoplus_i \big(\E_{V_{x, \theta}\cap V_{x, \eta}\bs D | V_{x, \theta}\cap V_{x, \eta}}
^{\varphi_i^{x, \eta}|_{V_{x, \theta}\cap V_{x, \eta}}}
\big)^\SA|_{W_{x, \theta}\cap W_{x, \eta}}.\]
Since $\{\varphi_i^{x, \theta}\}_i$ and $\{\varphi_i^{x, \theta}\}_i$
are good sets of irregular values,
we obtain equality
\[\varphi_i^{x, \theta}|_{V_{x, \theta}\cap V_{x, \eta}}=
\varphi_i^{x, \eta}|_{V_{x, \theta}\cap V_{x, \eta}}\]
by Proposition \ref{prop2.14}.
We set $V_x := \cup_{\theta\in\varpi^{-1}(x)} V_{x, \theta}$, 
then there exist a good set of irregular values $\{\varphi_i\}_i$ on $(V_x, V_x\cap D)$
such that $\varphi_i|_{V_{x, \theta}}=\varphi_i^{x, \theta}$ and 
\[(\M|_{V_{x}})^\SA|_{W_{x, \theta}}\simeq
\bigoplus_i \big(\E_{V_{x}\bs D | V_{x}}^{\varphi_i}\big)^\SA|_{W_{x, \theta}}.\]
\end{proof}

The following lemma means that
the enhanced solution functor $\Sol_X^{\rmE}$ induces
an equivalence of categories between
the full subcategory of $\ZEC_{\RR-c}(\I\CC_X)$ consisting of
objects which have a normal form
and the one of $\Mod(\D_X)$ consisting of
objects which have a normal form.
\begin{lemma}\label{lem3.3}
\begin{itemize}
\item[\rm (1)]
For any holonomic $\D_X$-module $\M$ which has a normal form along $D$,
the enhanced solution complex $K := \Sol_X^{\rmE}(\M)$ of $\M$
has a normal form along $D$.

\item[\rm (2)]
For any enhanced ind-sheaf $K\in\ZEC_{\RR-c}(\I\CC_X)$ which has a normal form along $D$,
$\RH_X^{\rmE}(K)$ is a holonomic $\D_X$-module which has a normal form along $D$
and there exists an isomorphism $$K\simto \Sol_X^{\rmE}\big(\RH_X^{\rmE}(K)\big).$$
\end{itemize}
\end{lemma}

\begin{proof}
\item[(1)]
Let $\M$ be a holonomic $\D_X$-module which has a normal form along $D$.
By Sublemma \ref{sublem3.2} (1),
it is enough to show that $K:=\Sol_X^{\rmE}(\M)$ satisfies the condition (iii)
in Definition \ref{def3.1}.

Since $\M$ has a normal form along $D$, by Sublemma \ref{sublem3.4},
for any $x\in D$, 
there exist an open neighborhood $V_x\subset X$ of $x$,
a good set of irregular values $\{\varphi_i^x\}$ on $(V_x, V_x\cap D)$ such that
for any $\theta\in\varpi^{-1}(x)$
there exists an an open neighborhood $W_{x, \theta}$ of $\theta$ such that 
\[
W_{x, \theta}\subset\varpi^{-1}(V_x) \mbox{ and }
(\SM|_{V_x})^\SA|_{W_{x, \theta}}
\simeq \Bigl(
\bigoplus_i\bigl(\SE_{V_x\bs D | V_x}^{\varphi_i^x}\bigr)^\SA\Bigr)|_{W_{x, \theta}}.\]
By Theorem \ref{thm2.13}, for any open sector $S_{x, \theta}\subset V_x\setminus D$
such that  $\var{\varpi^{-1}(S_{x, \theta})}\subset W_{x, \theta}$
we have an isomorphism \[\pi^{-1}\CC_{S_{x, \theta}}\otimes K|_{V_x}\simeq
\bigoplus_i \EE_{S_{x, \theta} | V_x}^{\Re\varphi_i^x}.\]
Therefore, we obtain an open neighborhood
$$U_x:=\cup_{\theta\in\varpi^{-1}(x)}\varpi\big(\Int\big(\var{\varpi^{-1}(S_{x,\theta})}\big)
\big)\subset V_x$$
of $x$ which satisfies $U_x\setminus D=\cup_{\theta\in\varpi^{-1}(x)}S_{x, \theta}$.
Since $\varpi^{-1}(x)$ is compact, 
we obtain a finite sectorial open covering
$\{U_{x, j}\}_j\subset \{S_{x, \theta}\}_{\theta\in\varpi^{-1}(x)}$ of $U_x\setminus D$.
Hence for any $x\in D$ there exist an open neighborhood $U_x\subset X$ of $x$,
a good set of irregular values $\{\varphi_i^x\}$ on $(U_x, U_x\cap D)$
and a finite sectorial open covering $\{U_{x, j}\}_j$ of $U_x\bs D$
such that
\[\pi^{-1}\CC_{U_{x, j}}\otimes K|_{U_x}\simeq
\bigoplus_i \EE_{U_{x, j} | U_x}^{\Re\varphi_i^x}.\]

\item[(2)]
Let us consider $\M := \RH_X^{\rmE}(K)\in\BDC(\D_X)$.
By Sublemma \ref{sublem3.2} (2),
$\M$ satisfies the conditions
$\M(\ast D)\simeq \M$ and $\singsupp(\M)\subset D$.
We shall prove that $\M$ is a holonomic $\D_X$-module which satisfies
the condition (iii) in Definition \ref{def2.8} and 
the canonical morphism $\Phi : K\to \Sol_X^{\rmE}(\RH_X^{\rmE}(K))$ is an isomorphism.

First, let us prove that $\M$ is holonoimc and
the canonical morphism $\Phi$ is an isomorphism.
Since $K$ satisfies the condition (ii) in Definition \ref{def3.1},
$\M$ is holonomic on $X\bs D$
and the restriction $\Phi|_{X\bs D}$ of $\Phi$ to $X\bs D$ is an isormorphism.
Hence, it is enough to prove that
for any $x\in D$ there exists an open neighborhood $U_x$ of $x$
such that $\M$ is holonomic on $U_x$ and
the restriction $\Phi|_{U_x}$ of $\Phi$ to $U_x$ is an isomorohism
by Proposition \ref{prop3.2} (see also Remark \ref{rem-stalk}).
Since $K$ satisfies the condition (iii) in Definition \ref{def3.1},
for any $x\in D$, there exist an open neighborhood $U_x\subset X$ of $x$,
a good set of irregular values $\{\varphi_i\}_i$ on $(U_x, U_x\cap D)$
and a finitely sectorial open covering $\{U_{x, j}\}_j$ of $U_x\bs D$
such that
\[\pi^{-1}\CC_{U_{x, j}}\otimes K|_{U_x}\simeq
\bigoplus_i \EE_{U_{x, j} | U_x}^{\Re\varphi_i}.\]
By \cite[Theorem 9.3, Lemma 9.8]{Mochi16},
we have $\M|_{U_x}\in\Conn(U_x; U_x\cap D)$
(in particular $\M|_{U_x}$ is holonomic)
and $\Phi|_{U_x} : K|_{U_x}\simto\Sol_X^\rmE(\M)|_{U_x}$. 
Therefore $\M$ is a holonomic $\D_X$-module and 
the canonical morphism $K\to \Sol_X^{\rmE}(\M)$ is an isomorphism.

We shall prove that $\M$ satisfies the condition (iii) in Definition \ref{def2.8}.
Since $K$ satisfies the condition (iii) in Definition \ref{def3.1},
we have isomorphisms
\begin{align*}
\pi^{-1}\CC_{U_{x, j}}\otimes\Sol_{U_x}^{\rmE}(\M|_{U_x})
&\simeq
\pi^{-1}\CC_{U_{x, j}}\otimes K|_{U_x}\\
&\simeq
\bigoplus_i \EE_{U_{x, j} | U_x}^{\Re\varphi_i}\\
&\simeq
\bigoplus_i \big(\pi^{-1}\CC_{U_{x, j}}\otimes\EE_{U_x\bs D | U_x}^{\Re\varphi_i}\big)\\
&\simeq
\bigoplus_i \big(\pi^{-1}\CC_{U_{x, j}}\otimes\Sol_{U_x}^{\rmE}(\E_{U_x \bs D | U_x}^{\varphi_i})\big)\\
&\simeq
\pi^{-1}\CC_{U_{x, j}}\otimes\Sol_{U_x}^{\rmE}
\big(\bigoplus_i \E_{U_x \bs D | U_x}^{\varphi_i}\big),
\end{align*}
where in the first (resp.\ forth) isomorphism
we used the fact $K\simto \Sol_X^{\rmE}(\M)$
(resp.\ Theorem \ref{thm2.5} (6)).
Let us denote by $\tl{U_x}$ the real blow-up of $U_x$ along $U_x\cap D$.
Then by Theorem \ref{thm2.12},
for any open subset $W\subset \tl{U_x}$ such that
$W\cap \varpi^{-1}(U_x\cap D)\neq\emptyset$,
$\var{W}\subset \Int(\var{\varpi^{-1}(U_{x, j})})$
and we have an isomorphism
\[(\M|_{U_x})^\SA|_W
\simeq
\big(\bigoplus_i \E_{U_x \bs D | U_x}^{\varphi_i}\big)^{\SA}|_W.\]
Hence the proof is completed.
\end{proof}

\subsection{Quasi-Normal Form}
In this subsection,
we define enhanced ind-sheaves 
which have a quasi-normal form along a normal crossing divisor
and prove that they are nothing but the images of holonomic $\D$-modules
which have a quasi-normal form via the enhanced solution functor.
Let $X$ be a complex manifold and $D$ a normal crossing divisor of $X$.

\begin{definition}\label{def3.4}
We say that an enhanced ind-sheaf
$K\in\ZEC(\I\CC_X)$ has a quasi-normal form along $D$ if 
\begin{itemize}
\setlength{\itemsep}{-3pt}
\item[(i)]
$\pi^{-1}\CC_{X\setminus D}\otimes K\simto K$,

\item[(ii)]
for any $x\in X\setminus D$, there exist an open neighborhood $U_x\subset X\setminus D$
of $x$ and a non-negative integer $k$ such that
\[K|_{U_x}\simeq (\CC_{U_x}^{\rmE})^{\oplus k},\]

\item[(iii)]
for any $x\in D$, there exist an open neighborhood $U_x\subset X$ of $x$
and a ramification $p_x : U_x'\to U_x$ of $U_x$ along $D_x := U_x\cap D$
such that $\bfE p_x^{-1}(K|_{U_x})$ has a normal form along $D_x' := p_x^{-1}(D_x)$.
\end{itemize}
\end{definition}

\begin{proposition}\label{prop3.7}
Any enhanced ind-sheaf which has a quasi-normal form along $D$
is an $\RR$-constructible enhanced ind-sheaf. 
\end{proposition}
\begin{proof}
Let $K\in\ZEC(\I\CC_X)$ be an enhanced ind-sheaf which has a quasi-normal form along $D$.
Since $K$ satisfies the condition (ii) in Definition \ref{def3.4}
and the constant enhanced ind-sheaf $\CC^\rmE$ is $\RR$-constructible,
for any $x\in X\setminus D$,
there exists an open neighborhood $U_x\subset X\setminus D$ of $x$
such that $K|_{U_x}\in\BEC_{\RR-c}(\I\CC_{U_x})$.
Since $K$ satisfies the condition (iii) in Definition \ref{def3.4},
for any $x\in D$, there exist an open neighborhood $U_x$ of $x$
and a ramification $p_x : U_x' \to U_x$
such that $\bfE p_x^{-1}(K|_{U_x})$ is $\RR$-constructible
because an enhanced ind-sheaf which has a normal form is $\RR$-constructible.
Since $p_x$ is proper, $\bfE p_{x\ast}\bfE p_x^{-1}(K|_{U_x})$
is also $\RR$-constructible by \cite[Theorem 4.9.11 (ii)]{DK16}.
Therefore $K|_{U_x}$ which is a direct summand
of $\bfE p_{x\ast}\bfE p_x^{-1}(K|_{U_x})$
is also $\RR$-constructible by \cite[Theorem 4.9.6]{DK16}.
Since the $\RR$-constructability of enhanced ind-sheaves is a local property,
the proof is completed.
\end{proof}

The following lemma means that
the enhanced solution functor $\Sol_X^{\rmE}$ induces
an equivalence of categories between
the full subcategory of $\ZEC_{\RR-c}(\I\CC_X)$ consisting of
objects which have a quasi-normal form
and the one of $\Mod(\D_X)$ consisting of
objects which have a quasi-normal form.
\begin{lemma}\label{lem3.6}
\begin{itemize}
\item[\rm (1)]
For any holonomic $\D_X$-module $\M$ which has a quasi-normal form along $D$,
the enhanced solution complex $K := \Sol_X^{\rmE}(\M)$ of $\M$
has a quasi-normal form along $D$.

\item[\rm (2)]
For any enhanced ind-sheaf $K\in\ZEC(\I\CC_X)$ which has a quasi-normal form along $D$,
$\RH_X^{\rmE}(K)$ is a holonomic $\D_X$-module which
has a quasi-normal form along $D$ and
there exists an isomorphism $$K\simto \Sol_X^{\rmE}\big(\RH_X^{\rmE}(K)\big).$$
\end{itemize}
\end{lemma}

\begin{proof}
(1) By Theorem \ref{thm2.5} (2), Sublemma \ref{sublem3.2}(1) and Lemma \ref{lem3.3}(1),
we obtain the assertion.

(2) Let us consider $\M := \RH_X^{\rmE}(K)\in\BDC(\D_X)$.
By Sublemma \ref{sublem3.2} (2),
$\M$ satisfies the conditions
$\M(\ast D)\simeq \M$ and $\singsupp(\M)\subset D$.
If $\M$ is a holonomic $\D_X$-module and 
the canonical morphism $\Phi : K\to \Sol_X^{\rmE}(\M)$ is an isomorphism,
$\M$ satisfies the third condition in Definition \ref{def2.9}
by Theorem \ref{thm2.5} (2), Lemma \ref{lem3.3} (2) and the fact that
$K$ satisfies the third condition in Definition \ref{def3.4}.
Hence, it is enough to prove that $\M$ is holonomic and
the canonical morphism $K\to \Sol_X^{\rmE}(\M)$ is an isomorphism.
Since $\M$ is holonomic on $X\bs D$ and
the restriction $\Phi|_{X\bs D}$ of $\Phi$ to $X\bs D$ is an isomorphism, 
we shall show that for any $x\in D$ there exists an open neighborhood $U_x$ of $x$
such that $\M$ is holonomic over $U_x$ and
the restriction $\Phi|_{U_x}$ of $\Phi$ to $U_x$ is an isomorphism.
By Lemma \ref{lem3.3} (2) and
the condition (iii) in the Definition \ref{def3.4}, 
for any $x\in D$ there exist an open neighborhood $U_x$ of $x$,
a ramification $p_x : U_x' \to U_x$ of $U_x$ along $D_x := U_x\cap D$
and a holonomic $\D_{U_x'}$-module $\N_{U_x'}$
which has a normal form along $D_x' := p_x^{-1}(D_x)$
such that $\bfE p_x^{-1}(K|_{U_x}) = \Sol_{U_x'}^{\rmE}(\N_{U_x'})$.
Then we obtain an isomorphism
\[\N_{U_x'} \simeq \rhom^{\rmE}\big(\bfE p_x^{-1}(K|_{U_x}), \SO_{U_x'}^\rmE\big)\]
by Theorem \ref{thm2.6} (1).
Moreover, we have isomorphisms
\begin{align*}
\bfD p_{x\ast}\N_{U_x'}
&\simeq
\bfD p_{x\ast}\rhom^{\rmE}\big(\bfE p_x^{-1}(K|_{U_x}), \SO_{U_x'}^\rmE\big)\\
&\simeq
\bfR p_{x\ast}\rhom^{\rmE}\big(\bfE p_x^{-1}(K|_{U_x}),
\bfE p_x^!\SO_{U_x}^\rmE\big)\\
&\simeq
\rhom^{\rmE}\big(\bfE p_{x\ast}\bfE p_x^{-1}(K|_{U_x}),\SO_{U_x}^\rmE\big)
\end{align*}
where in the second (resp.\ last) isomorphism
we used Sublemma \ref{sublem3.3} and \cite[Theorem 9.1.2 (i)]{DK16}
(resp.\ \cite[Lemma 4.5.17]{DK16}).
Hence $\M|_{U_x}$ is a direct summand of $\bfD p_{x\ast}\N_{U_x'}$
because $K|_{U_x}$ is a direct summand of $\bfE p_{x\ast}\bfE p_x^{-1}(K|_{U_x})$.
Since the morphism $p_x$ is proper,
$\bfD p_{x\ast}\N_{U_x'}$ is holonomic
by \cite[Theorem 4.4.1]{Sab11},
therefore $\M|_{U_x}$ is also holonomic.
We shall prove the restriction $\Phi|_{U_x}$ of $\Phi$ to $U_x$ is an isomorphism.
It is enough to show that
the canonical morphism $\bfE p_{x\ast}\bfE p_x^{-1}(K|_{U_x})\to
\Sol_{U_x}^\rmE\big(\RH_{U_x}^\rmE\big(\bfE p_{x\ast}\bfE p_x^{-1}(K|_{U_x})\big)\big)$
is an isomorphism
because $K|_{U_x}$ is a direct summand of $\bfE p_{x\ast}\bfE p_x^{-1}(K|_{U_x})$.
This follows from isomorphisms below:
\begin{align*}
\Sol_{U_x}^\rmE\big(\RH_{U_x}^\rmE\big(\bfE p_{x\ast}\bfE p_x^{-1}(K|_{U_x})\big)\big)
&\simeq
\Sol_{U_x}^\rmE\big(\rhom^{\rmE}\big(
\bfE p_{x\ast}\bfE p_x^{-1}(K|_{U_x}),\SO_{U_x}^\rmE\big)\big)\\
&\simeq
\Sol_{U_x}^\rmE(\bfD p_{x\ast}\N_{U_x'})\\
&\simeq
\bfE p_{x\ast}\Sol_{U_x'}^\rmE(\N_{U_x'})\\
&\simeq
\bfE p_{x\ast}\bfE p_x^{-1}(K|_{U_x})
\end{align*}
where the third (resp.\ last) isomorphism follows from Theorem \ref{thm2.5} (3)
(resp.\ the definition of $\N_{U_x'}$).
Hence the proof is completed.
\end{proof}

\subsection{Modified Quasi-Normal Form}
In this subsection, we define enhanced ind-sheaves
which have a modified quasi-normal form along an analytic hypersurface.
Moreover, we show that these are nothing but the images of meromorphic connections
via the enhanced solution functor.
Let $X$ be a complex manifold and $Y$ an analytic hypersurface of $X$.

\begin{definition}\label{def3.7}
We say that an enhanced ind-sheaf $K\in\ZEC(\I\CC_X)$
has a modified quasi-normal form along $Y$ if 
\begin{itemize}
\setlength{\itemsep}{-3pt}
\item[(i)]
$\pi^{-1}\CC_{X\setminus Y}\otimes K\simto K$,

\item[(ii)]
for any $x\in X\setminus Y$, there exist an open neighborhood $U_x\subset X\setminus Y$
of $x$ and a non-negative integer $k$ such that
\[K|_{U_x}\simeq (\CC_{U_x}^{\rmE})^{\oplus k},\]

\item[(iii)]
for any $x\in Y$, there exist an open neighborhood $U_x\subset X$ of $x$
and a modification $f_x : U_x'\to U_x$ of $U_x$ along $Y_x := U_x\cap Y$
such that $\bfE f_x^{-1}(K|_{U_x})$ has a quasi-normal form along $D_x' := f_x^{-1}(Y_x)$.
\end{itemize}
\end{definition}

\begin{proposition}\label{prop3.10}
Any enhanced ind-sheaf which has a modified quasi-normal form along $Y$
is an $\RR$-constructible enhanced ind-sheaf. 
\end{proposition}
\begin{proof}
By the condition (ii) in Definition \ref{def3.7},
for any $x\in X\setminus D$,
there exists an open neighborhood $U_x\subset X\setminus D$ of $x$
such that $K|_{U_x}\in\BEC_{\RR-c}(\I\CC_{U_x})$
because the constant enhanced ind-sheaf $\CC_{U_x}^\rmE$ is $\RR$-constructible.

By the condition (iii) in Definition \ref{def3.7},
for any $x\in Y$, there exist an open neighborhood $U_x$ of $x$
and a modification $f_x : U_x' \to U_x$
such that $\bfE f_x^{-1}(K|_{U_x})$ has a quasi-normal form along
$D_x' := f_x^{-1}(U_x\cap Y)$,
and hence it is $\RR$-constructible
by Proposition \ref{prop3.7}.
Since $f_x$ is proper, $\bfE f_{x\ast}\bfE f_x^{-1}(K|_{U_x})$ is also $\RR$-constructible
by \cite[Proposition 4.9.11 (ii)]{DK16}.
By the condition (i) in Definition \ref{def3.7}
and the fact that the modification $f_x$ induces an isomorphism
$U_x'\bs D_x'\simto U_x\bs Y$,
we have isomorphisms
\begin{align*}
\bfE f_{x\ast}\bfE f_x^{-1}(K|_{U_x})
&\simeq
\bfE f_{x\ast}\bfE f_x^{-1}\big((\pi^{-1}\CC_{X\bs Y}\otimes K)|_{U_x}\big)\\
&\simeq
\bfE f_{x\ast}\bfE f_x^{-1}(\pi^{-1}\CC_{U_x\bs Y}\otimes K|_{U_x})\\
&\simeq
\pi^{-1}\CC_{U_x\bs Y}\otimes K|_{U_x}
\simeq
K|_{U_x}.
\end{align*}
Therefore $K|_{U_x}$ is also $\RR$-constructible.
\end{proof}

The following lemma means that
the enhanced solution functor $\Sol_X^{\rmE}$ induces
an equivalence of categories between
the full subcategory of $\ZEC_{\RR-c}(\I\CC_X)$ consisting of
objects which have a modified quasi-normal form
and the abelian category $\Conn(X; Y)$ of
meromorphic connections on $X$ along $Y$.
\begin{lemma}\label{lem3.9}
\begin{itemize}
\item[\rm (1)]
For any meromorphic connection $\M$ on $X$ along $Y$,
the enhanced solution complex $K := \Sol_X^{\rmE}(\M)$ of $\M$
has a modified quasi-normal form along $Y$.

\item[\rm (2)]
For any enhanced ind-sheaf $K\in\ZEC(\I\CC_X)$
which has a modified quasi-normal form along $Y$,
$\RH_X^{\rmE}(K)$ is a meromorphic connection on $X$ along $Y$ 
which satisfies $$K\simto \Sol_X^{\rmE}\big(\RH_X^{\rmE}(K)\big).$$
\end{itemize}
\end{lemma}

\begin{proof}
(1) This follows from Corollary \ref{cor2.11},
Sublemma \ref{sublem3.2} (1) and Lemma \ref{lem3.6} (1).

(2) Let us consider $\M := \RH_X^{\rmE}(K)\in\BDC(\D_X)$.
It is clear that $\M$ satisfies the conditions
$\M(\ast Y)\simeq \M$ and $\singsupp(\M)\subset Y$
by Sublemma \ref{sublem3.2} (2).
We shall prove that $\M$ is a holonomic $\D_X$-module and 
the canonical morphism $\Phi : K\to \Sol_X^{\rmE}(\M)$ is an isomorphism.
Now, $\M$ is holonomic on $X\bs Y$ and
the restriction $\Phi|_{X\bs Y}$ of $\Phi$ to $X\bs Y$ is an isomorphism.
Hence, it is enough to show that for any $x\in Y$ there exists an open neighborhood $U_x$ of $x$
such that $\M$ is holonomic over $U_x$ and
the restriction $\Phi|_{U_x}$ of $\Phi$ to $U_x$ is an isomorphism.
By Lemma \ref{lem3.6} (2) and the fact that 
$K$ satisfies the condition (iii) in the Definition \ref{def3.7},
for any $x\in Y$ there exist an open neighborhood $U_x$ of $x$,
a modification $f_x : U_x' \to U_x$ of $U_x$ along $Y_x:=U_x\cap Y$
and a holonomic $\D_{U_x'}$-module $\N_{U_x'}$
which has a quasi-normal form along $D_x':=f_x^{-1}(Y_x)$ such that
$\bfE f_x^{-1}(K|_{U_x})=\Sol_{U_x'}^{\rmE}(\N_{U_x'})$. 
Then we obtain an isomorphism
\[\N_{U_x'} \simeq \rhom^{\rmE}\big(\bfE f_x^{-1}(K|_{U_x}), \SO_{U_x'}^\rmE\big)\]
by Theorem \ref{thm2.6}.
Moreover, we have isomorphisms
\begin{align*}
\bfD f_{x\ast}\N_{U_x'}
&\simeq
\bfD f_{x\ast}\rhom^{\rmE}\big(\bfE f_x^{-1}(K|_{U_x}), \SO_{U_x'}^\rmE\big)\\
&\simeq
\bfR f_{x\ast}\rhom^{\rmE}\big(\bfE f_x^{-1}(K|_{U_x}),
\bfE f_x^!\SO_{U_x}^\rmE\big)\\
&\simeq
\rhom^{\rmE}\big(\bfE f_{x\ast}\bfE f_x^{-1}(K|_{U_x}),\SO_{U_x}^\rmE\big)\\
&\simeq
\rhom^{\rmE}\big(K|_{U_x},\SO_{U_x}^\rmE\big)
\simeq \M|_{U_x},
\end{align*}
where in the second (resp.\ third) isomorphism
we used Sublemma \ref{sublem3.3} and \cite[Theorem 9.1.2 (i)]{DK16}
(resp.\ \cite[Lemma 4.5.17]{DK16})
and in the last one we used fact 
$\bfE f_{x\ast}\bfE f_x^{-1}(K|_{U_x})\simeq K|_{U_x}$.
Since the morphism $f_x$ is a modification,
$\bfD f_{x\ast}(\N_{U_x'})$ is also a meromorphic connection on $U_x$ along $Y_x$
by \cite[Proposition 8.16]{Sab13}
and therefore $\M|_{U_x}$ is also meromorphic connection (in particular, holonomic).
Moreover by the definition of $\N_{U_x'}$
we have isomorphisms
\begin{align*}
K|_{U_x}
&\simeq
\bfE f_{x\ast}\bfE f_x^{-1}(K|_{U_x})\\
&\simeq
\bfE f_{x\ast}\Sol_{U_x'}^{\rmE}(\N_{U_x'})\\
&\simeq
\Sol_{U_x}^{\rmE}(\bfD f_{x\ast}\N_{U_x'})\\
&\simeq
\Sol_{U_x}^\rmE(\M|_{U_x})
\simeq
\Sol_X^\rmE(\M)|_{U_x}
\end{align*}
where third isomorphism follows from Theorem \ref{thm2.5} (3)
and hence the morphism $\Phi|_{U_x}$ is an isomorphism.
The proof is completed.
\end{proof}

\begin{notation}\label{nota3.12}
We denote by $\ZECmero(\I\CC_{X(Y)})$ the essential image of
$$\Sol_X^{\rmE} : \Conn(X; Y)^{\op}\to\ZEC_{\RR-c}(\I\CC_X).$$
This abelian category is nothing but the full subcategory of $\ZEC_{\RR-c}(\I\CC_X)$
consisting of enhanced ind-sheaves which have a modified quasi-normal form along $Y$
by Lemma \ref{lem3.9}.
Moreover, we set
\begin{align*}
\BDCmero(\D_{X(Y)}) &:=\{\M\in\BDChol(\D_X)\
|\ \SH^i(\M)\in\Conn(X; Y) \mbox{ for any }i\in\ZZ \},\\
\BECmero(\I\CC_{X(Y)}) &:=\{K\in\BEC_{\RR-c}(\I\CC_X)\
|\ \SH^i(K)\in\ZECmero(\I\CC_{X(Y)}) \mbox{ for any }i\in\ZZ \}.
\end{align*}
\end{notation}
Since the category $\BDCmero(\D_{X(Y)})$ is a full triangulated subcategory
of $\BDChol(\D_X)$
and the category $\BECmero(\I\CC_{X(Y)})$ is a full triangulated subcategory
of $\BEC_{\RR-c}(\I\CC_X)$, the following proposition is obvious
by induction on the length of the complex:
\begin{proposition}\label{prop3.17}
The enhanced solution functor $\Sol_X^\rmE$ induces an equivalence of triangulated categories
$$\BDCmero(\D_{X(Y)})^{\op} \simto \BECmero(\I\CC_{X(Y)}),$$
and hence we obtain a commutative diagram
\[\xymatrix@C=30pt@M=3pt{
\BDCmero(\D_{X(Y)})^{\op}\ar@{->}[r]^\sim\ar@<1.0ex>@{}[r]^-{\Sol_X^{\rmE}}
\ar@{}[rd]|{\rotatebox[origin=c]{180}{$\circlearrowright$}}
 & \BECmero(\I\CC_{X(Y)})\\
\Conn(X; Y)^{\op}\ar@{->}[r]_-{\Sol_X^\rmE}^-{\sim}\ar@{}[u]|-{\bigcup}
&\ZECmero(\CC_{X(Y)}).\ar@{}[u]|-{\bigcup}
}\]
\end{proposition}

\subsection{$\CC$-Constructible Enhanced Ind-Sheaves}
In this subsection we define $\CC$-constructible enhanced ind-sheaves
and prove the main theorem.
Let $X$ be a complex manifold.

\begin{definition}\label{def3.19}
We say that an enhanced ind-sheaf $K\in\ZEC(\I\CC_X)$ is $\CC$-constructible if
there exists a complex stratification $\{X_\alpha\}_\alpha$ of $X$
such that $$\pi^{-1}\CC_{Z_\alpha\setminus D_\alpha}\otimes \bfE f_\alpha^{-1}K$$
has a modified quasi-normal form along $D_\alpha$ for any $\alpha$,
where $f_\alpha : Z_\alpha \to X$ is a complex blow-up of $\var{X_\alpha}$
along $\var{X_\alpha}\setminus X_\alpha$ and
$D_\alpha := f_\alpha^{-1}(\var{X_\alpha}\setminus X_\alpha)$.
Namely $Z_\alpha$ is a complex manifold,
$D_\alpha$ is a normal crossing divisor of $Z_\alpha$
and $f_\alpha$ is a projective map
which induces an isomorphism $Z_\alpha\setminus D_\alpha\simto X_\alpha$
and satisfies $f_\alpha(Z_\alpha)=\var{X_\alpha}$.

We call such a family $\{X_\alpha\}_{\alpha\in A}$ a stratification adapted to $K$.
\end{definition}

\begin{remark}\label{rem3.12}
\begin{itemize}
\item[(1)]
Definiton \ref{def3.19} does not depend on the choice of the complex blow-up $f_\alpha$
by Sublemma \ref{sublem3.22} below.

\item[(2)]
In the situation as above,
since $\pi^{-1}\CC_{Z_\alpha\setminus D_\alpha}\otimes \bfE f_\alpha^{-1}K$
has a modified quasi-normal form along $D_\alpha$,
there exists a meromorphic connection $\N_\alpha$ on $Z_\alpha$ along $D_\alpha$
such that $$\pi^{-1}\CC_{Z_\alpha\setminus D_\alpha}\otimes \bfE f_\alpha^{-1}K
\simto \Sol_{Z_\alpha}^{\rmE}(\N_\alpha)$$
by Lemma \ref{lem3.9} (2).
By applying the direct image functor $\bfE f_{\alpha !!}$
we obtain an isomorphism
$$\pi^{-1}\CC_{X_\alpha}\otimes K\simto
\Sol_X^{\rmE}(\bfD f_{\alpha\ast}\N_\alpha)[d_X-d_{X_\alpha}].$$
Moreover, by \cite[Theorem 4.4.1]{Sab11},
we have $\bfD f_{\alpha\ast}\N_\alpha\in\BDChol(\D_X)$
and hence $\pi^{-1}\CC_{X_\alpha}\otimes K\in\BEC_{\RR-c}(\I\CC_X)$.
\end{itemize}
\end{remark}

We denote by $\ZEC_{\CC-c}(\I\CC_X)$ the full subcategory of $\ZEC(\I\CC_X)$
whose objects are $\CC$-constructible
and set
\[\BEC_{\CC-c}(\I\CC_X) := \{K\in\BEC(\I\CC_X)\
|\ \SH^i(K)\in\ZEC_{\CC-c}(\I\CC_X) \mbox{ for any }i\in\ZZ \}\subset \BEC(\I\CC_X).\]

\begin{proposition}
The category $\ZEC_{\CC-c}(\I\CC_X)$ is
the full abelian subcategory of $\ZEC_{\RR-c}(\I\CC_X)$.
Hence the category $\BEC_{\CC-c}(\I\CC_X)$ is
a full triangulated subcategory of $\BEC_{\RR-c}(\I\CC_X)$.
\end{proposition}

\begin{proof}
First let us prove that 
the category $\ZEC_{\CC-c}(\I\CC_X)$ is abelian.
It is enough to show that the kernel and the cokernel of a morphism $\Phi : K\to L$
of $\CC$-constructible enhanced ind-sheaves are also $\CC$-constructible.
By Lemma \ref{lem3.14} below,
we can take a common stratification $\{X_\alpha\}_\alpha$ adapted to $K$ and $L$
with a common complex blow-up $f_\alpha : Z_\alpha\to X$ of $\var{X_\alpha}$
along $\var{X_\alpha}\bs X_\alpha$
such that there exist meromorphic connections $\M_\alpha, \N_\alpha$
on $Z_\alpha$ satisfying the following isomorphisms
\[\pi^{-1}\CC_{Z_\alpha\setminus D_\alpha}\otimes \bfE f_\alpha^{-1} K
\simto \Sol_{Z_\alpha}^{\rmE}(\M_\alpha), \hspace{7pt}
\pi^{-1}\CC_{Z_\alpha\setminus D_\alpha}\otimes \bfE f_\alpha^{-1}L
\simto \Sol_{Z_\alpha}^{\rmE}(\N_\alpha),\]
where we set $D_\alpha = f_\alpha^{-1}(\var{X_\alpha}\setminus X_\alpha)$.
Let $\varphi_\alpha :=
\RH_{Z_\alpha}^\rmE(
\pi^{-1}\CC_{Z_\alpha\bs D_\alpha}\otimes \bfE f_\alpha^{-1}\Phi)
 : \N_\alpha\to \M_\alpha$ be the morphism of meromorphic connections
induced by the morphism $\Phi : K\to L$.
Since the category of meromorphic connections is abelian,
$\Coker\varphi_\alpha$ is a meromorphic connection on $Z_\alpha$ along $D_\alpha$.
Moreover, we have $$\Ker\big(\Sol_{Z_\alpha}^\rmE(\varphi_\alpha)\big)
\simeq \Sol_{Z_\alpha}^{\rmE}(\Coker\varphi_\alpha)$$
because we have an equivalence of abelian categories
$$\Sol_{Z_\alpha}^\rmE : \Conn(Z_\alpha; D_\alpha)^{\op}
\simto\ZECmero(\I\CC_{Z_\alpha(D_\alpha)})$$
by Lemma \ref{lem3.9} (see also Notation \ref{nota3.12}).
Then we obtain a commutative diagram
\[\xymatrix@C=13pt@R=20pt{
0\ar@{->}[r] & \pi^{-1}\CC_{Z_\alpha\setminus D_\alpha}
\otimes\bfE f_\alpha^{-1}\Ker\Phi\ar@{->}[r]\ar@{.>}[d]_{{}^\exists}^-\wr
& \pi^{-1}\CC_{Z_\alpha\setminus D_\alpha}\otimes\bfE f_\alpha^{-1}K
\ar@{->}[r]\ar@{->}[d]^-\wr
& \pi^{-1}\CC_{Z_\alpha\setminus D_\alpha}\otimes\bfE f_\alpha^{-1}L
\ar@{->}[d]^-\wr\\
0\ar@{->}[r] & 
\Ker\big(\Sol_{Z_\alpha}^\rmE(\varphi_\alpha)\big)\ar@{->}[r]\ar@{-}[d]^-\wr
& \Sol_{Z_\alpha}^{\rmE}(\M_\alpha)
\ar@{->}[r]_-{\Sol_{Z_\alpha}^{\rmE}(\varphi_\alpha)}
& \Sol_{Z_\alpha}^{\rmE}(\N_\alpha).\\
{} & \Sol_{Z_\alpha}^{\rmE}(\Coker\varphi_\alpha) & {} & {}
}\]
Therefore, we have
$\pi^{-1}\CC_{Z_\alpha\setminus D_\alpha}\otimes\bfE f_\alpha^{-1}(\Ker\Phi)
\in\ZECmero(\I\CC_{Z_\alpha(D_\alpha)})$
and hence, $\Ker\Phi$ is $\CC$-constructible.
Similarly we can show that $\Coker\Phi$ is $\CC$-constructible.

Let us prove that any $\CC$-constructible enhanced ind-sheaf is $\RR$-constructible.
Let $\{X_\alpha\}_{\alpha\in A}$ be a stratification of $X$ adapted to $K$,
then $\pi^{-1}\CC_{X_\alpha}\otimes K$ is $\RR$-constructible,
see Remark \ref{rem3.12} (2).
Therefore by Lemma \ref{lem2.3},
for each $\alpha\in A$ there exist a locally finite family
$\{Z_\beta^\alpha\}_{\beta\in B_\alpha}$
of locally closed subanalytic subset of $X$ and
a family $\{\SF_\beta^\alpha\}_{\beta\in B_\alpha}$
of objects of $\BDC_{\RR-c}(\CC_{X\times\RR_\infty})$
such that 
\[\CC_X^{\rmE}\Potimes \SF_\beta^\alpha\simeq
\pi^{-1}\CC_{Z_\beta^\alpha}\otimes(\pi^{-1}\CC_{X_\alpha}\otimes K)
\simeq \pi^{-1}\CC_{X_\alpha\cap Z_\beta^\alpha}\otimes K.\]
Hence, the proof is completed by Lemma \ref{lem2.3}.
\end{proof}

\begin{sublemma}\label{sublem3.22}
Let $K\in\ZEC_{\CC-c}(\I\CC_X)$ and
$\{X_\alpha\}_{\alpha\in A}$ a stratification of $X$ adapted to $K$.
Then any stratification of $X$ which is finer than the one $\{X_\alpha\}_{\alpha\in A}$
is also adapted to $K$.
\end{sublemma}
 
\begin{proof}
Let $\{Y_\beta\}_{\beta\in B}$ be a stratification of $X$ finer than 
the one $\{X_\alpha\}_{\alpha\in A}$.
Then for each $\beta\in B$ there exists $\alpha\in A$
such that $Y_\beta\subset X_\alpha$.
Then we have the diagram:
\[\xymatrix@R=7pt@M=5pt{
Z_\alpha \ar@/^18pt/[rr]^{f_\alpha} \ar@{->}[r]&
\var{X_\alpha}\ar@{^{(}->}[r]\ar@{}[d]|-{\bigcup} & X\ar@{=}[d]\\
W_\beta\ar@{->}[r]\ar@/_18pt/[rr]_{g_\beta}
& \var{Y_\beta}\ar@{^{(}->}[r]  & X,
}\]
where $g_\beta : W_\beta\to X$ is a complex blow-up
of $\var{Y_\beta}$ along $\var{Y_\beta}\setminus Y_\beta$.
We set $H_\beta := g_\beta^{-1}(\var{Y_\beta}\setminus Y_\beta)$
then an enhanced ind-sheaf
$$\pi^{-1}\CC_{W_\beta\setminus H_\beta}\otimes\bfE g_\beta^{-1}K$$
is concentrated in degree zero
because the functor
$\pi^{-1}\CC_{W_\beta\setminus H_\beta}\otimes\bfE g_\beta^{-1}(\cdot)$
is t-exact with respect to the standard t-structure
(see \cite[Proposition 2.7.3 (iv) and Lemma 2.7.5 (i)]{DK16-2}).
Now there exists a meromorphic connection $\M_\alpha$ on $Z_\alpha$ along $D_\alpha$
such that $$\pi^{-1}\CC_{X_\alpha}\otimes
K\simeq\bfE f_{\alpha!!}\Sol_{Z_\alpha}^{\rmE}(\M_\alpha)$$
and hence we have a sequence of isomorphisms
\begin{align*}
\pi^{-1}\CC_{W_\beta\setminus H_\beta}\otimes\bfE g_\beta^{-1}K
&\simeq
\bfE g_\beta^{-1}(\pi^{-1}\CC_{Y_\beta}\otimes K)\\
&\simeq
\bfE g_\beta^{-1}\big(\pi^{-1}\CC_{Y_\beta}\otimes
(\pi^{-1}\CC_{X_\alpha}\otimes K)\big)\\
&\simeq
\bfE g_\beta^{-1}\big(\pi^{-1}\CC_{Y_\beta}\otimes
\bfE f_{\alpha!!}\Sol_{Z_\alpha}^{\rmE}(\M_\alpha)\big)\\
&\simeq
\pi^{-1}\CC_{W_\beta\setminus H_\beta}\otimes
\bfE g_\beta^{-1}\bfE f_{\alpha!!}\Sol_{Z_\alpha}^{\rmE}(\M_\alpha)\\
&\simeq
\pi^{-1}\CC_{W_\beta\setminus H_\beta}\otimes
\Sol_{W_\beta}^{\rmE}\big(\bfD g_{\beta}^{\ast}\bfD f_{\alpha\ast}(\M_\alpha)\big)
[d_X-d_{Z_\alpha}]\\
&\simeq
\Sol_{W_\beta}^{\rmE}\big(
(\bfD g_{\beta}^{\ast}\bfD f_{\alpha\ast}\M_\alpha)(\ast H_\beta)[d_{Z_\alpha}-d_X]\big),
\end{align*}
where in the fifth (resp.\ sixth) isomorphism
we used Theorem \ref{thm2.5} (2) and (3)
(resp.\ Theorem \ref{thm2.5} (5)). 
Let us set $$\N_\beta := (\bfD g_{\beta}^{\ast}\bfD f_{\alpha\ast}\M_\alpha)
(\ast H_\beta)[d_{Z_\alpha}-d_X]\in\BDC(\D_{W_\beta}).$$
Since $f_\alpha$ is proper
we have $\N_\beta\in\BDChol(\D_{W_\beta})$.
Moreover since $\N_\beta|_{W_\beta\bs H_\beta}$
is an integrable connection on $W_\beta\bs H_\beta$
we have $\N_\beta\in\BDCmero(W_\beta; H_\beta)$.
Since the enhanced ind-sheaf
$\pi^{-1}\CC_{W_\beta\setminus H_\beta}\otimes\bfE g_\beta^{-1}K$
is concentrated in degree zero,
we obtain
$\N_\beta\in\Conn(W_\beta; H_\beta)$
by Proposition \ref{prop3.17}.
Therefore we have 
$$\pi^{-1}\CC_{W_\beta\setminus H_\beta}\otimes\bfE g_\beta^{-1}K
\in\ZECmero(\I\CC_{W_\beta(H_\beta)})$$
and the proof is completed.
\end{proof}

By this sublemma, it is clear that the $\CC$-constructability is local property
and the following holds.  
Moreover, Definiton \ref{def3.19} does not depend on the choice of the complex blow-up $f_\alpha$.
\begin{lemma}\label{lem3.14}
For any two $\CC$-constructible enhanced ind-sheaves $K, L\in\ZEC_{\CC-c}(\I\CC_X)$,
there exist a common stratification $\{X_\alpha\}_\alpha$ adapted to $K$ and $L$
with a common complex blow-up of $\var{X_\alpha}$
along $\var{X_\alpha}\bs X_\alpha$.
\end{lemma}
\begin{proof}
This follows from Sublemma \ref{sublem3.22}.
\end{proof}

\begin{lemma}\label{lem3.23}
Let $\M$ be a holonomic $\D_X$-module.
Then there exists a stratification $\{X_\alpha\}_{\alpha\in A}$ of $X$
such that for any $\alpha\in A$ and
any complex blow-up $f_\alpha : Z_\alpha \to X$ of $\var{X_\alpha}$
along $\var{X_\alpha}\setminus X_\alpha$ we have
$(\bfD f_\alpha^\ast \M)(\ast D_\alpha)
\in\BDCmero(\D_{Z_\alpha(D_\alpha)})$,
where $D_\alpha := f_\alpha^{-1}(\var{X_\alpha}\setminus X_\alpha)$
is a normal crossing divisor of $Z_\alpha$.
\end{lemma}

\begin{proof}
First we shall construct a stratification $\{X_\alpha\}_{\alpha\in A}$
such that any cohomology of $\bfD i_{X_\alpha}^\ast(\M)$
is an integrable connection on $X_\alpha$ for each $\alpha\in A$.

We put $Y := \singsupp(\M)$.
Then $Y$ is an analytic subset of $X$ and
$\M|_{X\setminus Y}$ is an integrable connection on $X\setminus Y$ by definition.
Let us set
$$Y_1 := Y_{\sing}\cup \var{\singsupp(\bfD i_{Y_\reg}^{\ast}\M)}
=Y_{\sing}\cup \singsupp(\bfD i_{Y_\reg}^{\ast}\M)\subset Y.$$
Then $Y_1$ is an analytic subset of $X$ and
$Y\setminus Y_1\subset Y_{\reg}\setminus \singsupp(\bfD i_{Y_\reg}^{\ast}\M)$.
Hence, any cohomology of $\bfD i_{Y\setminus Y_1}^\ast\M$
is an integrable connection on $Y\setminus Y_1$.
Similarly to the construction as above,
we put $$Y_{k+1} := (Y_k)_{\sing}\cup \var{\singsupp(\bfD i_{(Y_k)_\reg}^{\ast}\M)}
= (Y_k)_{\sing}\cup \singsupp(\bfD i_{(Y_k)_\reg}^{\ast}\M)\subset Y_k.$$
Since $Y_k\setminus Y_{k+1}\subset (Y_k)_\reg\setminus\singsupp(\bfD i_{(Y_k)_\reg}^\ast\M)$,
any cohomology of $\bfD i_{Y_k\setminus Y_{k+1}}^\ast\M$ is an integrable connection on $Y_k\setminus Y_{k+1}$.
By $\dim Y_{k+1}<\dim Y_k$,
there exists a positive integer $m\in\NN$
such that $Y_{m+1}=\emptyset, Y_m\neq\emptyset$.
Namely, $Y_m$ is a smooth analytic subset of $X$ and 
any cohomology of $\bfD i_{Y_m}^{\ast}\M$ is an integrable connection on ${Y_m}$.
Therefore we obtain a partition
\[X = (X\setminus Y) \sqcup (Y\setminus Y_1)\sqcup \cdots
\sqcup (Y_{m-1}\setminus Y_m)\sqcup Y_m\]
and the desired stratification $\{X_\alpha\}_{\alpha\in A}$ finer than it.

Let $f_\alpha : Z_\alpha\to X$ be a complex blow-up of $\var{X_\alpha}$
along $\var{X_\alpha}\setminus X_\alpha$
and $D_\alpha := f_\alpha^{-1}(\var{X_\alpha}\setminus X_\alpha)$.
Since the restriction of $f_\alpha$ to $Z_\alpha\setminus D_\alpha$ induces
an isomorphism $Z_\alpha\setminus D_\alpha\simto X_\alpha$,
any cohomology of $(\bfD f_\alpha^\ast \M)|_{Z_\alpha\setminus D_\alpha}$
is an integrable connection on $Z_\alpha\setminus D_\alpha$.
Hence we obtain $(\bfD f_\alpha^\ast\M)(\ast D_\alpha)
\in \BDCmero(\D_{Z_\alpha(D_\alpha)})$
by \cite[Theorem 3.1]{Kas78}.
The proof is completed.
\end{proof}

\begin{proposition}
For any $\M\in\BDChol(\D_X)$
the enhanced solution complex $\Sol_X^{\rmE}(\M)$ of $\M$
is an object of $\BEC_{\CC-c}(\I\CC_X)$.
\end{proposition}
\begin{proof}
Since the category $\BDChol(\D_X)$ is a full triangulated subcategory of $\BDC(\D_X)$
and the category $\BEC_{\CC-c}(\I\CC_X)$ is a full triangulated subcategory
of $\BEC_{\RR-c}(\I\CC_X)$, 
it is enough to show the assertion in the case $\M\in\Modhol(\D_X)$ 
by induction on the lengths of the complexes.

Let $\M\in\Modhol(\D_X)$ and we put $K := \Sol_X^\rmE(\M)$.
By Lemma \ref{lem3.23}, 
there exist a stratification $\{X_\alpha\}_{\alpha\in A}$ of $X$
and a complex blow-up $f_\alpha : Z_\alpha \to X$ of $\var{X_\alpha}$
along $\var{X_\alpha}\setminus X_\alpha$ for each $\alpha\in A$
such that $(\bfD f_\alpha^\ast \M)(\ast D_\alpha)
\in\BDCmero(\D_{Z_\alpha(D_\alpha)})$,
where $D_\alpha := f_\alpha^{-1}(\var{X_\alpha}\setminus X_\alpha)$
is a normal crossing divisor.
Then we have
$$\pi^{-1}\CC_{Z_\alpha\setminus D_\alpha}\otimes \bfE f_\alpha^{-1}K
\simeq
\Sol_{Z_\alpha}^\rmE\big((\bfD f_\alpha^\ast \M)(\ast D_\alpha)\big)
\in\BECmero(\I\CC_{Z_\alpha(D_\alpha)})$$
for any $\alpha\in A$,
where we used Theorem \ref{thm2.5} (2), (5) and Proposition \ref{prop3.17}.
Since the functor
$\pi^{-1}\CC_{Z_\alpha\setminus D_\alpha}\otimes \bfE f_\alpha^{-1}(\cdot)$ is exact
we have 
\begin{align*}
\pi^{-1}\CC_{Z_\alpha\setminus D_\alpha}\otimes \bfE f_\alpha^{-1}(\SH^iK)
\simeq
\SH^i(\pi^{-1}\CC_{Z_\alpha\setminus D_\alpha}\otimes \bfE f_\alpha^{-1}K)
\in\ZECmero(\I\CC_{Z_\alpha(D_\alpha)})
\end{align*}
for any $i\in\ZZ$.
Therefore $K\in\BEC_{\CC-c}(\I\CC_X)$ and the proof is completed.
\end{proof}

By this proposition and the irregular Riemann-Hilbert correspondence
of A. D'Agnolo and M. Kashiwara
we obtain a fully faithful functor
\[\Sol_X^{\rmE} : \BDChol(\D_X)^{\op} \hookrightarrow \BEC_{\CC-c}(\I\CC_X).\]
We shall prove that this functor is essentially surjective.

\begin{theorem}\label{thm3.17}
For any $\CC$-constructible enhanced ind-sheaf $K\in\BEC_{\CC-c}(\I\CC_X)$,
there exists $\M\in\BDChol(\D_X)$
such that $$K\simto \Sol_X^{\rmE}(\M).$$
Therefore we obtain an equivalence of categories
\[\xymatrix@C=50pt@R=7pt{
\BDChol(\D_X)^{\op}\ar@<0.7ex>@{->}[r]^-{\Sol_X^{\rmE}}\ar@{}[r]|-{\sim}
&
\BEC_{\CC-c}(\I\CC_X)\ar@<0.7ex>@{->}[l]^-{\RH_X^{\rmE}}}.\]
\end{theorem}

\begin{proof}
By induction on the length of the complex,
it is enough to show in the case of $K\in\ZEC_{\CC-c}(\I\CC_X)$.
Let $\{X_\alpha\}_{\alpha\in A}$ be a stratification of $X$ adapted to $K$
and we put $$ Y_k := \bigsqcup_{\dim X_\alpha\leq k}X_\alpha,
\hspace{17pt}
S_k := Y_k\setminus Y_{k-1} = \bigsqcup_{\dim X_\alpha = k}X_\alpha
\hspace{17pt}
\mbox{for any } k=0, 1, \ldots, d_X.$$
Then  $Y_0=S_0$ and $X=Y_{d_X}$.
Moreover, there exists a distinguished triangle
\[\pi^{-1}\CC_{S_k}\otimes K\to
\pi^{-1}\CC_{Y_k}\otimes K \to
\pi^{-1}\CC_{Y_{k-1}}\otimes K\xrightarrow{+1}.\]
Hence, by induction on $k$,
it is enough to show that 
$\pi^{-1}\CC_{S_k}\otimes K\in\Sol_X^{\rmE}(\BDChol(\D_X))$ for any $k$.

Let $S_i$ be decomposed into $Z_1\sqcup\cdots \sqcup Z_{m_i}$
with some strata $Z_1, \ldots, Z_{m_i}\in\{X_\alpha\}_{\alpha\in A}$.
If $m_i=1$, by Remark \ref{rem3.12} (2)
we have $\pi^{-1}\CC_{S_i}\otimes K\in\Sol_X^{\rmE}(\BDChol(\D_X))$.
We shall prove the case $m_i\geq2$.
In this case there exists a distinguished triangle
\[\pi^{-1}\CC_{Z_1}\otimes K\to
\pi^{-1}\CC_{Z_1\sqcup\cdots \sqcup Z_j}\otimes K \to
\pi^{-1}\CC_{Z_2\sqcup\cdots \sqcup Z_j}\otimes K\xrightarrow{+1}\]
for any $j=2, \ldots, m_i$.
Hence, by induction on $j$
it is enough to show that 
$\pi^{-1}\CC_{Z_1}\otimes K\in\Sol_X^{\rmE}(\BDChol(\D_X))$.
However, it follows from Remark \ref{rem3.12} (2).
\end{proof}

By Theorem \ref{thm2.6} (2), we obtain:
\begin{corollary}\label{cor3.19}
The functor $e : \BDC(\CC_M) \hookrightarrow \BECstb(\I\CC_M)$
induces an embedding
$$\BDC_{\CC-c}(\CC_X)\hookrightarrow\BEC_{\CC-c}(\I\CC_X)$$
and hence we have a commutative diagram
\[\xymatrix@C=30pt@M=5pt{
\BDChol(\D_X)^{\op}\ar@{->}[r]^\sim\ar@<1.0ex>@{}[r]^-{\Sol_X^{\rmE}}
\ar@{}[rd]|{\rotatebox[origin=c]{180}{$\circlearrowright$}}
 & \BEC_{\CC-c}(\I\CC_X)\\
\BDCrh(\D_X)^{\op}\ar@{->}[r]_-{\Sol_X}^-{\sim}\ar@{}[u]|-{\bigcup}
&\BDC_{\CC-c}(\CC_X).
\ar@{^{(}->}[u]_-{e}
}\]
\end{corollary}

Moreover by Proposition \ref{prop2.7} and the fact that there exists an isomorphism
$\sh\big(\Sol_X^{\rmE}(\M)\big)\simeq \Sol_X(\M)$
for $\M\in\BDChol(\D_X)$, we have:

\begin{corollary}\label{cor3.22}
The functor $\sh : \BEC(\I\CC_X)\to\BDC(\CC_X)$ induces
$$\BEC_{\CC-c}(\I\CC_X)\to\BDC_{\CC-c}(\CC_X)$$
and hence we have a commutative diagram
\[\xymatrix@C=30pt@M=5pt{
\BDChol(\D_X)^{\op}\ar@{->}[r]^\sim\ar@<1.0ex>@{}[r]^-{\Sol_X^{\rmE}}
\ar@{->}[d]_-{(\cdot)_\reg}\ar@{}[rd]|{\rotatebox[origin=c]{180}{$\circlearrowright$}}
 & \BEC_{\CC-c}(\I\CC_X)\ar@{->}[d]^-{\sh}\\
\BDCrh(\D_X)^{\op}\ar@{->}[r]_-{\Sol_X}^-{\sim}&\BDC_{\CC-c}(\CC_X).
}\]
\end{corollary}

The $\CC$-constructability is closed under many operations. 
\begin{proposition}\label{prop3.18}
Let $f : X\to Y$ be a morphism of complex manifolds and
$K, K_1, K_2\in\BEC_{\CC-c}(\I\CC_X)$, $L\in\BEC_{\CC-c}(\I\CC_Y)$.
Then we have
\begin{itemize}
\item[\rm(1)]
$K_1\Potimes K_2, \Prihom(K_1, K_2)$ and $K\Pboxtimes L$ are $\CC$-constructible,

\item[\rm(2)]
$\rmD_X^{\rmE}(K)\in\BEC_{\CC-c}(\I\CC_X)$ and 
$K\simto \rmD_X^{\rmE}\rmD_X^{\rmE}K$,

\item[\rm(3)]
$\bfE f^{-1}L$ and $\bfE f^!L$ are $\CC$-constructible,

\item[\rm(4)]
if $f$ is proper $\bfE f_{!!}K (\simeq\bfE f_{\ast}K)$ is $\CC$-constructible.
\end{itemize}
\end{proposition}

\begin{proof}
Since the proofs of these assertions in the proposition are similar,
we only prove the first one of (3).

Let $f : X\to Y$ be a morphism of complex manifolds and $L\in\BEC_{\CC-c}(\I\CC_Y)$.
Then we have isomorphisms
$$\bfE f^{-1}L\simeq\bfE f^{-1}\Big( \Sol_Y^{\rmE}\big(\RH_Y^{\rmE}(L)\big)\Big)
\simeq\Sol_X^{\rmE}\Big(\bfD f^{\ast}\big( \RH_Y^{\rmE}(L)\big)\Big),$$
where in the second isomorphism we used Theorem \ref{thm2.5} (2).
Since $\bfD f^{\ast}\big( \RH_Y^{\rmE}(L)\big)\in\BDChol(\D_X)$
then we obtain $\bfE f^{-1}L\in\BEC_{\CC-c}(\I\CC_X)$.
\end{proof}

By this proposition, the functor $\RH^{\rmE}$ commutes with many operations as below
\begin{corollary}
Let $f : X\to Y$ be a morphism of complex manifolds and
$K, K_1, K_2\in\BEC_{\CC-c}(\I\CC_X)$, $L\in\BEC_{\CC-c}(\I\CC_Y)$.
Then
\begin{itemize}
\item[\rm(1)]
$\RH_X^{\rmE}(K_1\Potimes K_2)\simeq
\RH_X^{\rmE}(K_1)\Dotimes\RH_X^{\rmE}(K_2)$,
\item[\rm(2)]
$\RH_X^{\rmE}\big(\Prihom(K_1, K_2)\big)\simeq
\rhom_{\SO_X}\big(\RH_X^{\rmE}(K_1), \RH_X^{\rmE}(K_2)\big)$,
\item[\rm(3)]
$\RH_{X\times Y}^{\rmE}(K \Pboxtimes L)\simeq
\RH_X^{\rmE}(K)\Dboxtimes\RH_Y^{\rmE}(L)$,
\item[\rm(4)]
$\RH_X(\rmD_X^{\rmE}K)[2d_X]\simeq\DD_X(\RH_X^{\rmE}K)$,
\item[\rm(5)]
$\RH_X^{\rmE}(\bfE f^{-1}L)\simeq\bfD f^{\ast}\big(\RH_Y^{\rmE}(L)\big)$.
\end{itemize}
\end{corollary}

\begin{proof}
Since the proofs of these assertions in the corollary are similar,
we only prove (5).
Let $f : X\to Y$ be a morphism of complex manifolds and $L\in\BEC_{\CC-c}(\I\CC_Y)$.
Then we have a sequence of isomorphisms
\begin{align*}
\RH_X^{\rmE}(\bfE f^{-1}L) &\simeq
\RH_X^{\rmE}\Big(\bfE f^{-1}\Sol_Y^{\rmE}\big(\RH_Y^{\rmE}(L)\big)\Big)\\
&\simeq 
\RH_X^{\rmE}\Big(\Sol_Y^{\rmE}\big(
\bfD f^{\ast}\big(\RH_Y^{\rmE}(L)\big)\big)\Big)\\
&\simeq
\bfD f^{\ast}\big(\RH_Y^{\rmE}(L)\big),
\end{align*}
where the first and third isomorphisms follow from Theorem \ref{thm3.17}
and in the second isomorphism we used Theorem \ref{thm2.5} (2).
\end{proof}
On the other hand, we can prove that the functor $\RH_X^\rmE$ commutes with the direct image functors without assuming the $\CC$-constructability as follows:
\begin{proposition}
Let $f : X\to Y$ be a morphism of complex manifolds
and  $K\in\BEC_{\RR-c}(\I\CC_X)$ be an $\RR$-constructible enhanced ind-sheaf.
Then we have
$$\RH_Y^{\rmE}(\bfE f_{!!} K)[d_Y]\simeq \bfD f_{\ast}\big(\RH_X^{\rmE}(K)\big)[d_X].$$
\end{proposition}
\begin{proof}
Indeed, we have isomorphisms
\begin{align*}
\RH_Y^{\rmE}(\bfE f_{!!} K)
&=
\rhom^{\rmE}(\bfE f_{!!} K, \SO_Y^{\rmE})\\
&\simeq
\bfR f_\ast\rhom^{\rmE}(K, \bfE f^{!}\SO_Y^{\rmE})\\
&\simeq
\bfD f_{\ast}\big(\RH_X^{\rmE}(K)\big)[d_X-d_Y],
\end{align*}
where in the second (resp.\ last) isomorphism
we used \cite[Lemma 4.5.17]{DK16}
(resp.\ \cite[Theorem 9.1.2 (1)]{DK16} and Sublemma \ref{sublem3.3}).
\end{proof}

\section{Enhanced Perverse Ind-Sheaves}
In this section, we will define a t-structure on
the triangulated category $\BEC_{\CC-c}(\I\CC_X)$
of $\CC$-constructible enhanced ind-sheaves
so that its heart is equivalent to the abelian category $\Modhol(\D_X)$
of holonomic $\D_X$-modules.
Recall that $\sh := \alpha_Xi_0^!\bfR^{\rmE} : \BEC(\I\CC_X)\to\BDC(\CC_X)$
denotes the sheafification functor and
for any $\CC$-constructible enhanced ind-sheaf $K$,
its sheafification $\sh(K)$ is a  $\CC$-constructible sheaf by Corollary \ref{cor3.22}.

We denote by $\rmD_X :\BDC(\CC_X)^{\op}\to\BDC(\CC_X)$
the Verdier dual functor for sheaves,
see \cite[\S 3]{KS90} for the definition.
The sheafification functor $\sh : \BEC_{\CC-c}(\I\CC_X)\to \BDC_{\CC-c}(\CC_X)$ commutes
with the duality functor as follows.
\begin{lemma}\label{lem4.1}
For any $K\in\BEC_{\CC-c}(\I\CC_X)$
there exists an isomorphism
$$\sh\big(\rmD_X^{\rmE}(K)\big)\simeq\rmD_X\big(\sh(K)\big).$$
\end{lemma}

\begin{proof}
Let $K\in\BEC_{\CC-c}(\I\CC_X)$.
Then there exists an isomorphism $\Sol_X^{\rmE}\big(\RH_X^{\rmE}(K)\big)\simto K$
by Theorem \ref{thm3.17}. 
Therefore we have a sequence of isomorphisms
\begin{align*}
\sh\big(\rmD_X^{\rmE}(K)\big)
&\simeq
\sh\big(\rmD_X^{\rmE}\big(\Sol_X^{\rmE}(\RH_X^{\rmE}(K)\big)\big)\\
&\simeq
\sh\big(\Sol_X^{\rmE}\big(\DD_X(\RH_X^{\rmE}(K)\big)\big)[2d_X]\\
&\simeq
\Sol_X\big(\DD_X\big(\RH_X^{\rmE}(K)\big)\big)[2d_X]\\
&\simeq
\rmD_X\big(\Sol_X(\RH_X^{\rmE}(K))\big)\\
&\simeq
\rmD_X\big(\sh\big(\Sol_X^\rmE(\RH_X^{\rmE}(K))\big)\big)\\
&\simeq
\rmD_X(\sh(K)),
\end{align*}
where in the second isomorphism
we used Theorem \ref{thm2.5} (1),
in the third and fifth ones we used the isomorphism
$\sh\big( \Sol_X^{\rmE}(\M)\big)\simeq \Sol_X(\M)$
and in the forth one we used the isomorphism
$\Sol_X\big(\DD_X(\M)\big)[2d_X]
\simeq
\rmD_X\big(\Sol_X(\M)\big)$ for $\M\in\BDChol(\D_X)$.
\end{proof}

Let us recall the definition of perverse sheaves.
We consider the following full subcategories of $\BDC_{\CC-c}(\CC_X)$:
\begin{align*}
{}^p\bfD^{\leq0}_{\CC-c}(\CC_X) &:=
\{\SF\in\BDC_{\CC-c}(\CC_X)\ |\
\dim(\supp\SH^i\SF)\leq -i \hspace{7pt}
\mbox{for any } i\in\ZZ \},\\
{}^p\bfD^{\geq0}_{\CC-c}(\CC_X) &:=
\{\SF\in\BDC_{\CC-c}(\CC_X)\ |\ \rmD_X(\SF)\in{}^p\bfD^{\leq0}_{\CC-c}(\CC_X)\}\\
&=
\{\SF\in\BDC_{\CC-c}(\CC_X)\ |\
\dim(\supp\SH^i\rmD_X(\SF))\leq -i \hspace{7pt}
\mbox{for any } i\in\ZZ \}.
\end{align*}
Then $\big({}^p\bfD^{\leq0}_{\CC-c}(\CC_X), {}^p\bfD^{\geq0}_{\CC-c}(\CC_X)\big)$
is a t-structure on $\BDC_{\CC-c}(\CC_X)$.
We denote by $\Perv(\CC_X)$ the heart of its t-structure
and call an object of $\Perv(\CC_X)$ a perverse sheaf.

\begin{definition}
We define full subcategories of $\BEC_{\CC-c}(\I\CC_X)$ by
\begin{align*}
{}^p\bfE^{\leq0}_{\CC-c}(\I\CC_X) &:=
\{K\in\BEC_{\CC-c}(\I\CC_X)\ |\ \sh_X(K)\in{}^p\bfD^{\leq0}_{\CC-c}(\CC_X)\},\\
{}^p\bfE^{\geq0}_{\CC-c}(\I\CC_X) &:=
\{K\in\BEC_{\CC-c}(\I\CC_X)\ |\ \rmD_X^{\rmE}(K)\in{}^p\bfE^{\leq0}_{\CC-c}(\I\CC_X)\}\\
&=
\{K\in\BEC_{\CC-c}(\I\CC_X)\ |\ \sh_X(K)\in{}^p\bfD^{\geq0}_{\CC-c}(\CC_X)\}
\hspace{20pt}(\mbox{ by Lemma \ref{lem4.1}}).
\end{align*}
\end{definition}

The following fact was proved by
\cite[Theorem 4.1]{Kas75} and \cite[Theorem 3.5.1]{Bjo93}.
\begin{fact}\label{lem4.4}
For any $\M\in\BDChol(\D_X)$,
we have
\begin{itemize}
\item[\rm(1)]
$\M\in\DChol^{\leq0}(\D_X)\Longleftrightarrow
\Sol_X(\M)[d_X]\in{}^p\bfD^{\geq0}_{\CC-c}(\CC_X)$,
\item[\rm(2)]
$\M\in\DChol^{\geq0}(\D_X)\Longleftrightarrow
\Sol_X(\M)[d_X]\in{}^p\bfD^{\leq0}_{\CC-c}(\CC_X)$.
\end{itemize}
\end{fact}

\begin{theorem}\label{thm4.4}
For any $\M\in\BDChol(\D_X)$,
we have
\begin{itemize}
\item[\rm(1)]
$\M\in\DChol^{\leq0}(\D_X)\Longleftrightarrow
\Sol_X^{\rmE}(\M)[d_X]\in{}^p\bfE^{\geq0}_{\CC-c}(\I\CC_X)$,
\item[\rm(2)]
$\M\in\DChol^{\geq0}(\D_X)\Longleftrightarrow
\Sol_X^{\rmE}(\M)[d_X]\in{}^p\bfE^{\leq0}_{\CC-c}(\I\CC_X)$.
\end{itemize}

Therefore, the pair $\big({}^p\bfE^{\leq0}_{\CC-c}(\I\CC_X),
{}^p\bfE^{\geq0}_{\CC-c}(\I\CC_X)\big)$
is a t-structure on $\BEC_{\CC-c}(\I\CC_X)$
and
its heart $$\Perv(\I\CC_X) :=
{}^p\bfE^{\leq0}_{\CC-c}(\I\CC_X)\cap{}^p\bfE^{\geq0}_{\CC-c}(\I\CC_X)$$ is equivalent to the abelian category $\Modhol(\D_X)$
of holonomic $\D_X$-modules.
\end{theorem}

\begin{proof}
Since there exists an isomorphism $\sh\big(\Sol_X^{\rmE}(\M)\big)\simeq \Sol_X(\M)$
for any $\M\in\BDChol(\D_X)$,
this theorem follows from Fact \ref{lem4.4} . 
\end{proof}

Let us recall that
the triangulated category $\BEC_{\RR-c}(\I\CC_X)$
has generalized t-structures
$$\big({}_{\frac{1}{2}}\bfE_{\RR-c}^{\leq c}(\I\CC_X),
{}_{\frac{1}{2}}\bfE_{\RR-c}^{\geq c}(\I\CC_X)\big)_{c\in\RR}
\hspace{7pt}\mbox{and}\hspace{7pt}
\big({}^{\frac{1}{2}}\bfE_{\RR-c}^{\leq c}(\I\CC_X),
{}^{\frac{1}{2}}\bfE_{\RR-c}^{\geq c}(\I\CC_X)\big)_{c\in\RR}$$
such that for any $c\in\ZZ$ we have
\begin{align*}
{}^{\frac{1}{2}}\bfE^{\leq c}_{\RR-c}(\I\CC_M)
\subset
{}_{\frac{1}{2}}\bfE^{\leq c}_{\RR-c}(\I\CC_M)
&\subset
\bfE^{\leq c}_{\RR-c}(\I\CC_M),\\
\bfE^{\geq c}_{\RR-c}(\I\CC_M)
\subset
{}_{\frac{1}{2}}\bfE^{\geq c}_{\RR-c}(\I\CC_M)
&\subset
{}^{\frac{1}{2}}\bfE^{\geq c}_{\RR-c}(\I\CC_M)
\end{align*}
by \cite[Lemma 3.2.3, Lemma 3.4.4 and (3.5.1)]{DK16-2}
(see also \S 2.5).
Moreover, for any $c\in \RR$ we have
\begin{align*}
\Sol_X^{\rmE}\Big(\DChol^{\leq c}(\D_X)\Big)[d_X] &\subset
{}_{\frac{1}{2}}\bfE^{\geq -c}_{\RR-c}(\I\CC_X)\subset
{}^{\frac{1}{2}}\bfE^{\geq -c}_{\RR-c}(\I\CC_X),\\
\Sol_X^{\rmE}\Big(\DChol^{\geq c}(\D_X)\Big)[d_X] &\subset
{}^{\frac{1}{2}}\bfE^{\leq -c}_{\RR-c}(\I\CC_X),\\
\RH_X^{\rmE}\Big({}^{\frac{1}{2}}\bfE^{\leq c}_{\RR-c}(\I\CC_X)\Big)[d_X]
&\subset\bfD^{\geq -c}(\D_X).
\end{align*}
by \cite{DK16-2} (see also \S 2.5 and Theorem \ref{thm2.6} (3)).

\begin{corollary}\label{cor4.6}
We have
\begin{align*}
{}^p\bfE_{\CC-c}^{\leq 0}(\I\CC_X) &=
{}^{\frac{1}{2}}\bfE_{\RR-c}^{\leq 0}(\I\CC_X)\cap
\BEC_{\CC-c}(\I\CC_X),\\
{}^p\bfE_{\CC-c}^{\geq 0}(\I\CC_X) &=
{}_{\frac{1}{2}}\bfE_{\RR-c}^{\geq 0}(\I\CC_X)\cap
\BEC_{\CC-c}(\I\CC_X)\\
&={}^{\frac{1}{2}}\bfE_{\RR-c}^{\geq 0}(\I\CC_X)\cap
\BEC_{\CC-c}(\I\CC_X).
\end{align*}

\end{corollary}

\begin{proof}
Let us only check that
\begin{align*}
{}^p\bfE_{\CC-c}^{\geq 0}(\I\CC_X) &=
{}_{\frac{1}{2}}\bfE_{\RR-c}^{\geq 0}(\I\CC_X)\cap
\BEC_{\CC-c}(\I\CC_X)
={}^{\frac{1}{2}}\bfE_{\RR-c}^{\geq 0}(\I\CC_X)\cap
\BEC_{\CC-c}(\I\CC_X).
\end{align*}
First, we remark that,
since ${}_{\frac{1}{2}}\bfE_{\RR-c}^{\geq 0}(\I\CC_X)
\subset{}^{\frac{1}{2}}\bfE_{\RR-c}^{\geq 0}(\I\CC_X)$,
we have 
$${}_{\frac{1}{2}}\bfE_{\RR-c}^{\geq 0}(\I\CC_X)\cap
\BEC_{\CC-c}(\I\CC_X)
\subset{}^{\frac{1}{2}}\bfE_{\RR-c}^{\geq 0}(\I\CC_X)\cap
\BEC_{\CC-c}(\I\CC_X).$$

If $K\in{}^p\bfE_{\CC-c}^{\geq 0}(\I\CC_X)$
then $\RH_X^\rmE(K)\in\DChol^{\leq d_X}(\D_X)$
by Theorem \ref{thm4.4} (1) and hence,
$$K\simto\Sol_X^\rmE\big(\RH_X^\rmE(K)\big)
\in{}_{\frac{1}{2}}\bfE_{\RR-c}^{\geq 0}(\I\CC_X)$$
by Theorem \ref{thm2.6} (3).
Namely, we obtain $K\in{}_{\frac{1}{2}}\bfE_{\RR-c}^{\geq 0}(\I\CC_X)\cap
\BEC_{\CC-c}(\I\CC_X)$.

On the other hand,
if $K\in{}^{\frac{1}{2}}\bfE_{\RR-c}^{\geq 0}(\I\CC_X)\cap
\BEC_{\CC-c}(\I\CC_X)$
then $$\rmD_X^\rmE(K)
\in{}^{\frac{1}{2}}\bfE_{\RR-c}^{\leq 0}(\I\CC_X)\cap\BEC_{\CC-c}(\I\CC_X)$$
by the definition of the generalized t-structure 
$\big({}^{\frac{1}{2}}\bfE_{\RR-c}^{\leq c}(\I\CC_X),
{}^{\frac{1}{2}}\bfE_{\RR-c}^{\geq c}(\I\CC_X)\big)_{c\in\RR}$
on $\BEC_{\RR-c}(\I\CC_X)$
and Proposition \ref{prop3.18} (2).
By Theorem \ref{thm2.6} (3), 
we have $\RH_X^\rmE(\rmD_X^\rmE K)\in\DChol^{\geq d_X}(\D_X)$
and hence,
$$\sh\big(\rmD_X^\rmE(K)\big)\simto\Sol_X\big(\RH_X^\rmE(\rmD_X^\rmE K)\big)
\in{}^p\bfD^{\leq0}_{\CC-c}(\CC_X)$$
by Theorem \ref{thm4.4} (2).
Namely, we obtain $K\in{}^p\bfE_{\CC-c}^{\geq 0}(\I\CC_X)$.
\end{proof}

\begin{definition}
We say that
$K\in\BEC_{\CC-c}(\I\CC_{X})$ is a enhanced perverse ind-sheaf
 if $K\in\Perv(\I\CC_{X}) := 
 {}^p\bfE^{\leq0}_{\CC-c}(\I\CC_X)\cap{}^p\bfE^{\geq0}_{\CC-c}(\I\CC_X)$.
\end{definition}

By the definition of the t-structure
$\big({}^p\bfE^{\leq0}_{\CC-c}(\I\CC_X), {}^p\bfE^{\geq0}_{\CC-c}(\I\CC_X)\big)$,
we have
\begin{align*}
\rmD_X^\rmE(K)\in{}^p\bfE^{\geq0}_{\CC-c}(\I\CC_X) &
\mbox{ for } K\in{}^p\bfE^{\leq0}_{\CC-c}(\I\CC_X),\\
\rmD_X^\rmE(K)\in{}^p\bfE^{\leq0}_{\CC-c}(\I\CC_X) &
\mbox{ for } K\in{}^p\bfE^{\geq0}_{\CC-c}(\I\CC_X).
\end{align*}

Thus, the functor $\rmD_X^\rmE$ induces an equivalence of categories
$$\rmD_X^{\rmE} : \Perv(\I\CC_X)^{\op}\simto\Perv(\I\CC_X).$$
Since $\id\simto\sh \circ e$, by Corollary \ref{cor3.19}
we obtain:
\begin{proposition}
\begin{itemize}
\item[\rm(1)]
For any $\SF\in{}^p\bfD_{\CC-c}^{\leq0}(\CC_X)$,
we have $e(\SF)\in{}^p\bfE_{\CC-c}^{\leq0}(\I\CC_X)$.
\item[\rm(2)]
For any $\SF\in{}^p\bfD_{\CC-c}^{\geq0}(\CC_X)$,
we have $e(\SF)\in{}^p\bfE_{\CC-c}^{\geq0}(\I\CC_X)$.
\item[\rm(3)] 
The functor $e : \BDC_{\CC-c}(\CC_X)\hookrightarrow\BEC_{\CC-c}(\I\CC_X)$
induces an embedding
\[\Perv(\CC_X)\hookrightarrow\Perv(\I\CC_X).\]
\end{itemize}
\end{proposition}

Note that the sheafification functor $\sh : \BEC(\I\CC_X)\to\BDC(\CC_X)$ commutes
with the direct image functor.
Indeed, let $f : X\to Y$ be a morphism of complex manifolds and $K\in\BECstb(\I\CC_X)$.
Then we have a sequence of isomorphisms
\begin{align*}
\sh_Y(\bfE f_\ast K) &\simeq
\rhom^\rmE(\CC_{\{t\geq0\}}\oplus\CC_{\{t\leq 0\}}, \bfE f_\ast K)\\
&\simeq 
\bfR f_\ast\rhom^\rmE\big(
\bfE f^{-1}(\CC_{\{t\geq0\}}\oplus\CC_{\{t\leq 0\}}), K\big)\\
&\simeq
\bfR f_\ast\rhom^\rmE(\CC_{\{t\geq0\}}\oplus\CC_{\{t\leq 0\}}, K)\\
&\simeq
\bfR f_\ast\big(\sh_X(K)\big),
\end{align*}
where in the second isomorphism we used \cite[Lemma 4.5.17]{DK16}.

\begin{proposition}
Let $f : X\to Y$ be a proper morphism of complex manifolds.
We assume that there exists a non-negative integer $d\in \ZZ_{\geq0}$
such that $\dim f^{-1}(y)\leq d$ for any $y\in Y$.
Here $\dim f^{-1}(y)$ is the dimension of $f^{-1}(y)$ as an analytic space. 
Then we have
\begin{itemize}
\item[\rm(1)]
for any $K\in{}^p\bfE_{\CC-c}^{\leq0}(\I\CC_X)$,
we have $\bfE f_{!!} K\in{}^p\bfE_{\CC-c}^{\leq d}(\I\CC_Y)$,
\item[\rm(2)]
for any $K\in{}^p\bfE_{\CC-c}^{\geq0}(\I\CC_X)$,
we have $\bfE f_{!!} K\in{}^p\bfE_{\CC-c}^{\geq -d}(\I\CC_Y)$.
\end{itemize}
\end{proposition}

\begin{proof}
Since the proofs of these assertions in the proposition are similar,
we only prove the assertion (1).
If $\SF\in{}^p\bfD_{\CC-c}^{\leq0}(\CC_X)$
then $\bfR f_{!} K\in{}^p\bfD_{\CC-c}^{\leq d}(\CC_Y)$
by the assumptions (see, e.g., \cite[Proposition 8.1.42]{HTT08}).
Since $\sh(K)\in{}^p\bfD_{\CC-c}^{\leq0}(\CC_X)$,
we have
$\sh(\bfE f_{!!} K)\simeq\bfR f_{!} \big(\sh(K)\big)\in{}^p\bfD_{\CC-c}^{\leq d}(\CC_Y)$
and hence $\bfE f_{!!} K\in{}^p\bfE_{\CC-c}^{\leq d}(\I\CC_Y)$.
\end{proof}

By using the properties of the $\RR$-constructible enhanced ind-sheaves, we have:
\begin{proposition}
Let $f : X\to Y$ be a morphism of complex manifolds.
We assume that there exists a non-negative integer $d\in \ZZ_{\geq0}$
such that $\dim f^{-1}(y)\leq d$ for any $y\in Y$.
\begin{itemize}
\item[\rm(1)]
For any $L\in{}^p\bfE_{\CC-c}^{\leq0}(\I\CC_Y)$
we have $\bfE f^{-1} L\in{}^p\bfE_{\CC-c}^{\leq d}(\I\CC_X)$.
\item[\rm(2)]
For any $L\in{}^p\bfE_{\CC-c}^{\geq0}(\I\CC_Y)$
we have $\bfE f^! L\in{}^p\bfE_{\CC-c}^{\geq -d}(\I\CC_X)$.
\end{itemize}
\end{proposition}

\begin{proof}
we only prove the assertion (1).
For any $L\in{}^p\bfE_{\CC-c}^{\leq0}(\I\CC_Y)$
we have $L\in{}^{\frac{1}{2}}\bfE_{\RR-c}^{\leq0}(\I\CC_Y)$.
Hence, by \cite[Proposition 3.5.6]{DK16-2}
we obtain $\bfE f^{-1} L\in{}^{\frac{1}{2}}\bfE_{\RR-c}^{\leq d}(\I\CC_X)$.
Moreover since $\bfE f^{-1} L$ is $\CC$-constructible by Proposition \ref{prop3.18},
we have $\bfE f^{-1} L\in{}^p\bfE_{\CC-c}^{\leq d}(\I\CC_X)$.
\end{proof}


\begin{thebibliography}{99}
\bibitem[BBD]{BBD}
A. A. Beilinson, J. Bernstein and P. Deligne,
Faisceaux pervers. In: Analysis and topology on singular spaces,
I (Luminy, 1981), Ast{\'e}risque 100, Soc. Math. France, Paris, 1982, 5-171.


\bibitem[Bj{\"o}93]{Bjo93}
Jan-Erik Bj{\"o}rk,
Analytic $D$-modules and applications, volume 247 of {
  Mathematics and its Applications},
Kluwer Academic Publishers Group, Dordrecht, 1993.

\bibitem[DK16]{DK16}
Andrea D'Agnolo and Masaki Kashiwara,
Riemann-{H}ilbert correspondence for holonomic {D}-modules,
{Publ. Math. Inst. Hautes {\'E}tudes Sci.}, 2016, 123:69--197.

\bibitem[DK19]{DK16-2}
Andrea D'Agnolo and Masaki Kashiwara,
Enhanced perversities, J. Reine Angew. Math. (Crelle's Journal) 751, 2019, 185-241.

\bibitem[HTT08]{HTT08}
Ryoshi Hotta, Kiyoshi Takeuchi, and Toshiyuki Tanisaki,
{{$D$}-modules, perverse sheaves, and representation theory},
  volume 236 of {Progress in Mathematics},
  Birkh{\"a}user Boston, 2008.

\bibitem[Ito20]{Ito20}
Yohei Ito, Note On The Algebraic Irregular Riemann--Hilbert Correspondence,
arXiv:2004.13518, preprint.

\bibitem[IT18]{IT18}
Yohei Ito and Kiyoshi Takeuchi,
On irregularities of Fourier transforms of 
regular holonomic {D}-Modules, 
Adv. Math. (2020), in press.

\bibitem[Kas75]{Kas75}
Masaki Kashiwara,
On the maximally overdetermined system of linear differential equations, I,
Publ. RIMS, Kyoto Univ., 10, 1975, 563-579.

\bibitem[Kas78]{Kas78}
Masaki Kashiwara,
On the holonomic systems of linear differential equations, II,
Inventiones math., 49, 1978, 121-135.

\bibitem[Kas84]{Kas84}
Masaki Kashiwara,
The Riemann-Hilbert problem for holonomic systems,
Publ. Res. Inst. Math. Sci., 20, no. 2, 1984, 319-365.

\bibitem[Kas16]{Kas15}
Masaki Kashiwara,
Self-dual t-structure, Publ. Res. Inst. Math. Sci., 52, no. 3, 2016, 271-295.

\bibitem[KK]{KK}
Kashiwara, M. and Kawai, T.,
On holonomic systems of microdifferential equations
III—system with regular singularities, Publ. RIMS, Kyoto Univ., 17, 1981, 813-979.

\bibitem[KS90]{KS90}
Masaki Kashiwara and Pierre Schapira,
{Sheaves on manifolds}, volume 292 of {Grundlehren der
  Mathematischen Wissenschaften},
Springer-Verlag, Berlin, 1990.

\bibitem[KS01]{KS01}
Masaki Kashiwara and Pierre Schapira,
Ind-sheaves,
{Ast{\'e}risque}, (271):136, 2001.

\bibitem[KS06]{KS06}
Masaki Kashiwara and Pierre Schapira,
{Categories and Sheaves}, volume 332 of {Grundlehren der
  Mathematischen Wissenschaften},
Springer-Verlag Berlin Heidelberg, 2006.

\bibitem[KS16a]{KS16-2}
Masaki Kashiwara and Pierre Schapira,
 Irregular holonomic kernels and {L}aplace transform,
{Selecta Math.}, 22(1), 2016, 55-109.

\bibitem[KS16b]{KS16}
Masaki Kashiwara and Pierre Schapira,
{Regular and irregular holonomic {D}-modules}, volume 433 of {
  London Mathematical Society Lecture Note Series},
   Cambridge University Press, Cambridge, 2016.

\bibitem[Ked10]{Ked10}
Kiran~S. Kedlaya,
Good formal structures for flat meromorphic connections, {I}:
  surfaces,
  {Duke Math. J.}, 154(2), 2010, 343-418.

\bibitem[Ked11]{Ked11}
Kiran~S. Kedlaya,
Good formal structures for flat meromorphic connections, {II}:
 excellent schemes,
  {J. Amer. Math. Soc.}, 24(1), 2011, 183-229.
  
  \bibitem[Kuwa18]{Kuwa18}
  Tatsuki Kuwagaki,
  Irregular perverse sheaves, arXiv:1808.02760, preprint.

  \bibitem[Mal94a]{Mal94}
  Bernard Malgrange,
  “Connexions m{\'e}romorphes”, in Congr{\'e}s Singularit{\'e}s (Lille, 1991),
  Cambridge University Press, 1994, 251-261.
 
  \bibitem[Mal94b]{Mal94-2}
  Bernard Malgrange,
  “Filtration des modules holonomes”,
  in Analyse alg{\'e}brique des perturbations singuli{\'e}res
  (L. Boutet de Monvel, ed.), Travaux en cours, vol. 48, no. 2, Hermann, Paris, 1994, 35-41.
  
   \bibitem[Mal96]{Mal96}
  Bernard Malgrange,
  “Connexions m{\'e}romorphes, II: le r{\'e}seau canonique”,
  Invent. Math. 124, 1996, p. 367-387.

\bibitem[Moc09]{Mochi09}
Takuro Mochizuki,
Good formal structure for meromorphic flat connections on smooth
  projective surfaces,
In {Algebraic analysis and around}, volume~54 of {Adv. Stud.
  Pure Math.}, 2009, 223--253.

\bibitem[Moc11]{Mochi11}
Takuro Mochizuki,
Wild harmonic bundles and wild pure twistor {$D$}-modules,
{Ast{\'e}risque}, (340), 2011.

\bibitem[Moc16]{Mochi16}
Takuro Mochizuki,
Curve test for enhanced ind-sheaves and holonomic d-modules,
arXiv:1610.08572, preprint.

\bibitem[Sab11]{Sab11}
Claude Sabbah,
Introduction to  the theory of $\D$-modules,
Lecture Notes, Nankai, 2011.

\bibitem[Sab13]{Sab13}
Claude Sabbah,
Introduction to Stokes structures,
Lect. Notes in Math., vol. 2060, Springer-Verlag, 2013.

\bibitem[SKK]{SKK}
Sato Mikio, Kashiwara Masaki and Kawai Takahiro,
Hyperfunctions and pseudo differential equations, Lecture Notes in Math. 287, Springer, 1973, 265-529.

\end{thebibliography}
\end{document}